\documentclass[a4paper,11pt]{amsart}
\usepackage{amsfonts,amssymb}
\usepackage{amsmath}
\usepackage{amsthm}
\usepackage[all]{xy}
\usepackage{enumerate}
\usepackage{relsize}
\usepackage{mathrsfs}
\usepackage[top=1.5in,bottom=1.3in,left=1.2in,right=1.2in]{geometry}
\usepackage[colorlinks=true,citecolor=blue,urlcolor=blue,linkcolor=blue]{hyperref}

\newtheorem{theorem}{Theorem}[section]
\newtheorem{lemma}[theorem]{Lemma}
\newtheorem{corollary}[theorem]{Corollary}
\newtheorem{proposition}[theorem]{Proposition}

\newtheorem*{coro}{Corollary}
\newtheorem*{prop}{Proposition}

\newtheorem{innercustomthm}{Theorem}
	\newenvironment{customthm}[1]
  {\renewcommand\theinnercustomthm{#1}\innercustomthm}
  {\endinnercustomthm}

\theoremstyle{definition}
\newtheorem{definition}[theorem]{Definition}
\newtheorem{example}[theorem]{Example}
\newtheorem{question}[theorem]{Question}

\theoremstyle{remark}
\newtheorem{remark}[theorem]{Remark}

\numberwithin{equation}{section}

\newcommand{\mc}{\mathsmaller{\mathsmaller{\geq0}}}
\newcommand{\mec}{\mathsmaller{\mathsmaller{>0}}}
\newcommand{\rank}[1]{\textnormal{rk}\:#1}
\newcommand{\set}[1]{\left\{#1\right\}}
\newcommand{\cen}{\textnormal{c}}
\newcommand{\spe}{\textnormal{sp}}
\newcommand{\rec}{\textnormal{r}}
\newcommand{\sop}[1]{\mathlarger{\mathlarger{\mathfrak s}}_#1}
\newcommand{\C}[1]{\text{Core}\left(#1\right)}
\newcommand{\RZ}[1]{\textnormal{RZ}#1}
\newcommand{\NL}[1]{\textnormal{NL}#1}
\newcommand{\Val}{\mathcal{V}}

\newcommand{\RZtilde}{\widetilde{\textnormal{RZ}}}
\newcommand{\esq}{\Sigma_{X'}}

\title{Topology of spaces of valuations and geometry of singularities}

\author{Ana Bel\'en de Felipe}

\address{A.B. de Felipe: BCAM -- Basque Center for Applied Mathematics. 
Alameda de Mazarredo, 14. E-48009 Bilbao, Basque Country -- Spain.}

\email{adefelipe@bcamath.org}

\date{}

\thanks{This research was supported by ERCEA Consolidator Grant 615655 -- NMST; 
the Basque Government through the BERC 2014--2017 program; the Spanish Ministry of Economy 
and Competitiveness MINECO: BCAM Severo Ochoa excellence accreditation SEV--2013--0323 and 
MTM2016-80659-P; and the ACIISI (with a co-financing rate of 85\% from ESF)
}

\keywords{Riemann--Zariski space, normalized non-Archimedean link, valuative tree, dual graph}

\subjclass[2010]{14B05,14E15}

\begin{document}
 
\begin{abstract}
Given an algebraic variety $X$ defined over an algebraically closed field, we study the space 
$\textnormal{RZ}{(X,x)}$ consisting of all the valuations of the function field of $X$ which 
are centered in a closed point $x$ of $X$. We concentrate on its homeomorphism type. We prove 
that, when $x$ is a regular point, this homeomorphism type only depends on the dimension of $X$. 
If $x$ is a singular point of a normal surface, we show that it only depends on the dual graph 
of a good resolution of $(X,x)$ up to some precise equivalence. This is done by studying the relation 
between $\textnormal{RZ}{(X,x)}$ and the norma\-lized non-Archimedean link of $x$ in $X$ coming 
from the point of view of Berkovich geometry. We prove that their behavior is the same.
\end{abstract}

\maketitle 


\section{Introduction}
\label{intro}

In this work we suppose that $X$ is an algebraic variety defined over an algebraically closed 
field $k$ and we fix a closed point $x$ in $X$. We initiate the study of the homeomorphism type 
of the space $\RZ{(X,x)}$ consisting of all valuation rings of the function field of $X$ dominating 
the local ring $\mathcal{O}_{X,x}$, endowed with the topology induced by the Zariski topology. 
We call $\RZ(X,x)$ the Riemann--Zariski space of $X$ at $x$. 

 Our goal is to clarify the relation between the topological properties of this space 
and the local geometry of $X$ at $x$. Note that the one-dimensional case is well understood: if 
$X$ is an algebraic curve then $\RZ(X,x)$ is in bijection with the set of the local analytic 
branches of $X$ at $x$. However, the situation is richer in higher dimension.

 As we shall explain in this introduction, similar considerations have appeared in the context 
of the theory of analytic spaces as developed by Berkovich and others after \cite{Ber}. Adopting 
this point of view one associates to $X$ its analytification $X^\text{an}$. A point of $X^\text{an}$ 
is an absolute value on the residue field of a point of $X$, extending the trivial absolute value of 
$k$. We may consider the subspace $\hbox{L}(X,x)$ of all points in $X^\text{an}$ which specialize 
to $x$ excepting the trivial one and then identify points defining equivalent valuations. We 
obtain in this way the normalized non-Archimedean link $\NL{(X,x)}$ of $x$ in $X$ (see \cite{FanC}).

 We shall refer to these spaces as spaces of valuations. Note that they have different 
topological properties. For instance, $\NL{(X,x)}$ is a compact space whereas $\RZ{(X,x)}$ 
is not Hausdorff in general. We clarify the relation between them: 

\begin{prop}[Propositions~\ref{pi} and \ref{quotientHausdorff}]
There exists a canonical continuous surjective map from $\RZ{(X,x)}$ to $\NL{(X,x)}$. 
Moreover, when $(X,x)$ is a normal surface singularity, $\NL{(X,x)}$ is the largest Hausdorff 
quotient of $\RZ{(X,x)}$
\end{prop}
 
 We do this by first showing that $\NL{(X,x)}$ can be viewed as the set of normalized 
semivaluations of $\mathcal{O}_{X,x}$ equipped with the topology of pointwise convergence. 
By a semivaluation on a ring $A$ we mean here a map from $A$ to $[0,+\infty]$ verifying the 
standard axioms of valuations but which may send to infinity some non-zero elements of $A$. 
We also give a sketch of a proof that the last statement in the previous proposition 
may fail in higher dimension (see Example~\ref{ejdimtres}).

 We address first the regular case. Our main result in this case is the following:

\begin{customthm}{A}[Theorems~\ref{regNL} and \ref{firstlemma}]\label{thmregular}
Let $x\in X$, $y\in Y$ be regular closed points of two algebraic varieties defined over 
the same algebraically closed field $k$. The following statements are equivalent:
\begin{enumerate}[(1)]
\item The spaces $\RZ{(X,x)}$ and $\RZ{(Y,y)}$ are homeomorphic.
\item The spaces $\mathrm{NL}(X,x)$ and $\mathrm{NL}(Y,y)$ are homeomorphic.
\item The varieties $X$ and $Y$ have the same dimension.
\end{enumerate}
\end{customthm}

 Our approach is as follows. We recover the dimension of $X$ from the Krull dimension of 
$\RZ{(X,x)}$ and the covering dimension of $\NL{(X,x)}$. Then in the case of $\NL{(X,x)}$ 
we show the uniqueness of the extension of a semivaluation of $\mathcal{O}_{X,x}$ to its 
formal completion. In the Riemann-Zariski setting the proof is more involved. We rely on 
\cite[Theorem~7.1]{HOST} that allows to extend valuations to the henselization of 
$\mathcal{O}_{X,x}$ in a canonical way.

 Two consequences of this statement are particularly noteworthy. On the one hand, 
this result shows that the homeomorphism type of $\RZ{(X,x)}$ and $\NL{(X,x)}$ depends only 
on the dimension of the variety $X$. In dimension two, one can be more specific. A topological 
model for $\hbox{NL}(\mathbf{A}_{\mathbf{C}}^2,0)$ has already been proposed in 
\cite[Section~3.2.3]{VT}. In this monograph $\NL{(\mathbf{A}_{\mathbf{C}}^2,0)}$ is referred 
to as the valuative tree. This space carries a canonical affine structure which allows 
one to perform convex analysis on it and which finds interesting applications in dynamics and 
complex analysis in \cite{FJ1,FJ}. More precisely it has a rooted nonmetric $\mathbf R$-tree 
structure (see Subsection~\ref{subsectrees}). Roughly speaking, this means that it is a 
topological space where any two different points are joined by a unique real line interval. 
This structure was extended in \cite{Gr} to the case of a regular closed point of a surface. 
The homeomorphism type of an arbitrary Berkovich curve is also studied in \cite{HLP} under 
a countability assumption on the base field. Since $\hbox{NL}(\mathbf{A}_{k}^2,0)$ is 
homeomorphic to the closure of the open unit ball in the Berkovich analytification of the 
affine line over the discrete valued field $k(\!(t)\!)$ (see \cite{VT}), the results of 
\cite{HLP} show that $\hbox{NL}(\mathbf{A}_{k}^2,0)$ is a Wa\.{z}ewski universal dendrite 
when $k$ is countable.

 On the other hand, assuming that resolution of singularities holds, Theorem~\ref{thmregular} 
reveals the self-homeomorphic structure of $\RZ{(X,x)}$ by considering a projective system of 
non-singular varieties. This property is also observed in $\hbox{NL}(\mathbf{A}_{\mathbf{C}}^2,0)$ 
and in the space of real places of $L(y)$ where $L=\mathbf{R}(\!(t^{\mathbf Q})\!)$ (see 
\cite[Theorem~6.51]{VT} and \cite[Corollary~21]{KK} respectively). More precisely, we obtain: 

\begin{coro}[Corollary~\ref{RZisautosim}]
Let $X$ be an algebraic variety defined over an algebraically closed field $k$ 
of characteristic zero. If $x\in X$ is a regular closed point, then for any open 
subset $U\subseteq\RZ{(X,x)}$ there is a subset $V\subseteq U$ such that $V$ is 
homeomorphic to $\RZ{(X,x)}$.
\end{coro}

 Next, we consider a singular point $x$ of a normal algebraic surface $X$. The dual graph 
associated to a good resolution of $(X,x)$ becomes a fundamental tool for our purpose. Note 
that any two such graphs are equivalent in the following sense: they can be made isomorphic by 
subdividing edges and attaching trees. A nice way to describe this equivalence is by looking at 
the core of a graph in the sense of \cite{Sta}. 

By a graph we mean a finite connected graph with at least one vertex. From a graph $\Gamma$ 
we can build in a natural way a topological space $|\Gamma|$ called its topological realization. 
The core of a graph $\Gamma$ which is not a tree is the subgraph of $\Gamma$ obtained by 
repeatedly deleting a vertex of degree one and the edge incident to it, until no vertex of 
degree one remains. In other words, the core $\C{\Gamma}$ of the graph $\Gamma$ is 
the smallest subgraph of $\Gamma$ with the same homotopy type. By convention we define the 
core of a tree to be the empty set and we set $|\emptyset|:=\emptyset$. We say that two graphs 
$\Gamma$ and $\Gamma'$ are equivalent if $|\C{\Gamma}|$ and $|\C{\Gamma'}|$ are homeomorphic.

 Our main result in this case is the following: 

\begin{customthm}{B}\label{thmsuperficies}
Let $x\in X$ and $y\in Y$ be singular points of normal algebraic surfaces defined over the 
same algebraically closed field $k$, and let $\Gamma_{X'}$ and $\Gamma_{Y'}$ be the dual graphs 
associated to good resolutions $\pi_{X'}:X'\to X$ and $\pi_{Y'}:Y'\to Y$ of $(X,x)$ and 
$(Y,y)$ respectively. The following statements are equivalent:
\begin{enumerate}[(1)]
\item The spaces $\RZ{(X,x)}$ and $\RZ{(Y,y)}$ are homeomorphic.
\item The spaces $\mathrm{NL}(X,x)$ and $\mathrm{NL}(Y,y)$ are homeomorphic.
\item The graphs $\Gamma_{X'}$ and $\Gamma_{Y'}$ are equivalent.
\end{enumerate}
\end{customthm}

 Two important ingredients of the proof of this result are the structure of the valuative 
tree of \cite{VT}, and the properties of the core of $\NL{(X,x)}$, which corresponds to what 
Berkovich called the skeleton in \cite{Ber}. 

 Observe that this statement implies that the spaces of valuations $\RZ{(X,x)}$ and $\NL{(X,x)}$ 
associated to any rational surface singularity $(X,x)$ are homeomorphic to 
$\RZ{(\mathbf{A}_k^2,0)}$ and $\NL{(\mathbf{A}_k^2,0)}$ respectively. The converse is not true, 
as the example of a cone over an elliptic curve shows. See Example~\ref{explane} for details. 

 In order to obtain more precise information on the singularity $(X,x)$ it will be necessary to 
explore finer structures of $\RZ{(X,x)}$. Actually both spaces of valuations are locally ringed 
spaces. In \cite{Fan}, $\NL{(X,x)}$ is endowed with a natural analytic structure locally modeled 
on affinoid spaces over $k(\!(t)\!)$. However this structure contains too much information for 
our purposes since by \cite[Corollary~4.14]{Fan} it determines the completion of the local ring 
${\mathcal O}_{X,x}$.

Finally, Theorem~B also shows that there exist normal surface singularities such that their 
normalized non-Archimedean links are homotopy equivalent but not homeomorphic (see Remark~\ref{NLnothomeo}).

The rest of the paper is organized as follows. In Section~\ref{section1} we introduce the spaces 
of valuations we deal with in this work and discuss some of its topological features. The first 
subsection is devoted to the Riemann--Zariski space and the second one to the normalized 
non-Archimedean link. Section~\ref{section1} ends with the study of the relationship between 
these spaces. The purpose of Section~\ref{section2} is to give the proof of Theorem~\ref{thmregular}. 
Section~\ref{section3} provides a short discussion of trees and graphs, which appear as 
essential tools in the next section. Finally, Section~\ref{section4} contains the proof of 
Theorem~\ref{thmsuperficies}.

 From now on we denote by $X$ an algebraic variety defined over an algebraically closed field 
$k$, i.e.\ $X$ is an integral separated scheme of finite type over $k$. We denote by $K$ its 
function field and by $d$ its dimension. We assume that $d>0$.

\subsection*{Acknowledgements}
This work has greatly benefited of many discussions with Charles Favre and Bernard Teissier. 
I would like to thank them and also Johannes Nicaise and the referee for many comments and 
suggestions. I am also grateful to Universidad de La Laguna, Universit\'e 
de Versailles Saint-Quentin--en--Yvelines and Institut de Math\'ematiques de Jussieu--Paris 
Rive Gauche for their hospitality.


\section{Spaces of valuations}\label{section1}

\subsection{The Riemann--Zariski space of \texorpdfstring{$\boldsymbol{X}$}{X} at 
\texorpdfstring{$\boldsymbol{x}$}{x}}

 Let $\mathfrak{X}$ be the set of all valuations of $K$ extending the trivial valuation 
of $k$, equipped with the \emph{Zariski topology}. Given $\nu\in\mathfrak{X}$, we denote by 
$\rank{\nu}$ the rank of its value group, $R_\nu$ its valuation ring, $m_\nu$ the maximal ideal 
of $R_\nu$ and $k_\nu$ the residue field. This topology is obtained by taking the subsets 
\[E(A)=\set{\nu\in\mathfrak{X}\:/\:A\subset R_\nu},\] 
where $A$ ranges over the family of all finite subsets of $K$, as a basis of open sets. If 
$A=\set{f_1,\ldots,f_m}$ then we write $E(A)=E(f_1,\ldots,f_m)$. Classically, two valuations 
of $K$ which differ by an order-preserving group isomorphism are called equivalent and they are 
identified. 

\emph{The Riemann--Zariski space} $\mathfrak{X}$ of $K|k$ was first introduced by Zariski, 
who esta\-blished its quasi-compactness in \cite{Za}. This result was a key point in his approach 
to the problem of resolution of singularities (this turns out to be also a key result in 
some recent attempts to solve this problem in positive characteristic following new strategies 
also using local uniformization, see \cite{Co,Te14}). However, observe that $\mathfrak{X}$ is 
in general very far from being Hausdorff: given $\nu\in\mathfrak{X}$, its closure in 
$\mathfrak{X}$ is the set of all valuations $\nu'$ in $\mathfrak{X}$ such that 
$R_{\nu'}\subseteq R_\nu$ (see \cite[Ch.~VI, \S 17, Theorem~38]{ZS}).

 Given a valuation $\nu\in\mathfrak{X}$, by the valuative criterion of separatedness, there exists 
at most one scheme-theoretic point $\xi\in X$ such that $R_\nu$ \emph{dominates} the local ring 
$\mathcal{O}_{X,\xi}$, that is $\mathcal{O}_{X,\xi}\subseteq R_\nu$ and $m_\nu$ intersects 
$\mathcal{O}_{X,\xi}$ in its maximal ideal $\mathfrak{m}_{X,\xi}$. If such a point exists, it is 
called the \emph{center} of $\nu$ in $X$ (we may also consider the closure of $\set{\xi}$ in $X$ 
as the center) and we say that $\nu$ is \emph{centered} in $X$. We call $\text{RZ}(X)$ the set of 
all valuations $\nu\in\mathfrak{X}$ such that $\nu$ is centered in $X$, endowed with the topology 
induced by the Zariski topology of $\mathfrak{X}$. 

When $X$ is affine, saying that $\nu$ belongs to $\text{RZ}(X)$ is equivalent to saying that $R_\nu$ 
contains the ring of regular functions on $X$. In general, if $\set{U_i}_{i=1}^m$ is an open affine 
covering of $X$, then $\text{RZ}(X)$ is the (finite) union of the open sets $E(A_i)$, where $A_i$ 
is a finite subset of $\mathcal O_X(U_i)$ generating $\mathcal O_X(U_i)$ as $k$-algebra. Since any 
basic open subset of $\mathfrak{X}$ is in fact quasi-compact (this follows from \cite[Ch.~VI, \S 17, 
Theorem~40]{ZS}), we conclude that $\text{RZ}(X)$ is a quasi-compact open subset of $\mathfrak{X}$. 
In addition, as a consequence of the valuative criterion of properness, it is equal to the whole 
Riemann--Zariski space if and only if $X$ is a complete variety. Finally, let us define the 
\emph{center map} $\cen_X:\text{RZ}(X)\to X$, which sends any valuation of $\text{RZ}(X)$ to 
its center in $X$. If $X$ is affine, then the inverse image of the basic open set of $X$ defined 
by a non-zero regular function $f$ on $X$ is $E(f^{-1})\cap\text{RZ}(X)$. The center map is 
continuous.
 
 A birational morphism $\pi_{X'}:X'\rightarrow X$ induces an isomorphism between the function 
fields of $X$ and $X'$, so we can identify their Riemann--Zariski spaces. Moreover, if $\nu$ is 
an element of $\text{RZ}(X')$ then it also belongs to $\text{RZ}(X)$ and $\pi_{X'}(\cen_{X'}(\nu))=
\cen_X(\nu)$. In the case that $\pi_{X'}$ is also proper, then it follows from the valuative 
criterion of properness that $\text{RZ}(X')$ and $\text{RZ}(X)$ coincide. Consider the projective 
system consisting of all proper birational models $(X',\pi_{X'})$ over $X$, and endow its projective 
limit $\mathcal Z$ with the projective limit topology (i.e.\ the coarsest topology for which 
all the projection maps $\mathcal Z\to X'$ are continuous). A fundamental theorem of Zariski 
states that the natural map from $\textnormal{RZ}(X)$ to $\mathcal Z$, which corresponds to 
sending a valuation $\nu$ to its center $\cen_{X'}(\nu)$ on each model $Y$, is a homeomorphism.

\begin{definition}
Given a closed point $x\in X$, we denote by $\RZ{(X,x)}$ the set of all valuations 
of $\text{RZ}(X)$ whose center in $X$ is $x$, equipped with the induced topology. We call 
this space the Riemann--Zariski space of $X$ at $x$. 
\end{definition}

 Note that $\RZ{(X,x)}$ is a closed subspace of $\text{RZ}(X)$ (since it is the fiber over $x$ 
of the continuous map $\cen_X$) and it is therefore itself quasi-compact. Let us first explore 
the case of curves.

\begin{remark}
Suppose $d=1$ and fix a closed point $x\in X$. Let $\bar x_1,\ldots,\bar x_m$ be the closed points 
in the fiber over $x$ of the normalization morphism $\overline{X}\rightarrow X$ (observe that 
by \cite[7.8.3-(vii)]{EGAIV}, $m$ is the number of local analytic branches of $X$ at $x$). Then  
$\RZ{(X,x)}=\bigsqcup_{1\leq i\leq m}\RZ{(\overline{X},\bar x_i)}$. Take $i\in\set{1,\ldots,m}$. 
Since $d=1$, $\bar x_i$ is a regular point of $\overline{X}$. Then 
$\mathcal{O}_{\overline{X},\bar x_i}$ is a regular local ring of dimension one with fraction 
field $K$ and therefore a valuation ring of $K$. The space $\RZ{(\overline{X},\bar x_i)}$ is 
thus reduced to a unique point. From this we deduce that $\RZ{(X,x)}$ is a finite set. 
\end{remark}

 Recall that the \emph{Krull dimension} of a topological space $Z$ is the supremum 
in the extended real line of the lengths of all chains of irreducible closed subspaces of $Z$. 
A chain 
$$\emptyset\subsetneq Z_0\subsetneq\ldots\subsetneq Z_l\subseteq Z$$ is of length $l$. We 
denote by $\dim Z$ the Krull dimension of $Z$.

\begin{proposition}\label{dimension}
For any closed point $x\in X$, $\dim\RZ{(X,x)}=d-1$.
\end{proposition}

\begin{proof}
We have already pointed out that $\dim\RZ{(X,x)}=0$ when $X$ is a curve. Suppose now 
that $d>1$.

 Let $\pi:\widetilde{X}\rightarrow X$ be the normalization of the blowing-up of $X$ at 
the closed point $x\in X$. Let us fix an irreducible component $E_1$ of $\pi^{-1}(x)$. We call 
$\nu_1$ the divisorial valuation defined by $E_1$. Note that $\text{tr.deg}_k k_{\nu_1}=d-1>0$. 
We choose a valuation 
$\overline{\nu}_1$ of the residue field $k_{\nu_1}$ such that $k\subseteq R_{\overline{\nu}_1}$ 
and $\text{tr.deg}_k k_{\overline{\nu}_1}=d-2$. If we denote by $p_1:R_{\nu_1}\rightarrow k_{\nu_1}$ 
the canonical projection, then $p_1^{-1}(R_{\overline{\nu}_1})$ is a valuation ring of $K$ 
dominating $\mathcal{O}_{X,x}$ and contained in $R_{\nu_1}$. We call $R_{\nu_2}$ this valuation 
ring. Using the classical notation, we have $\nu_2=\nu_1\circ\overline{\nu}_1$. The valuation 
$\nu_2$ belongs to the closure of $\nu_1$ in $\RZ{(X,x)}$ and 
$\text{tr.deg}_k k_{\nu_2}=\text{tr.deg}_k k_{\overline{\nu}_1}=d-2$. If $d-2>0$, then we can 
repeat this construction defining similarly $\overline{\nu}_2,p_2$ and $\nu_3$. 

 Observe that we may iterate the above construction until we have built a sequence of 
composite valuations $\nu_{i+1}=\nu_i\circ\overline{\nu}_i$ for $1\leq i\leq d-1$ in $\RZ{(X,x)}$ 
such that 
\begin{equation}\label{chain}
\emptyset\subsetneq\overline{\set{\nu_d}}\subseteq\overline{\set{\nu_{d-1}}}\subseteq
\ldots\subseteq\overline{\set{\nu_1}}\subseteq\RZ{(X,x)},\tag{$\star$}
\end{equation}
\noindent where the bar means closure in $\RZ{(X,x)}$. Since $\set{\nu_i}$ is irreducible in 
$\RZ{(X,x)}$ so is its closure. We now observe that two different valuations have 
different closures. This follows from the fact that $\RZ{(X,x)}$ is a Kolmogorov space. 
Indeed, if $\nu$ and $\nu'$ are two distinct valuations of $\RZ{(X,x)}$, then either 
$R_\nu\not\subset R_{\nu'}$ or $R_{\nu'}\not\subset R_\nu$. Without loss of generality 
we may suppose that there exists $f\in K$ such that $f\in R_\nu$ and $f\notin R_{\nu'}$. 
Then $E(f)\cap\RZ{(X,x)}$ is an open set containing $\nu$ but not containing $\nu'$, 
and the $T_0$ axiom is satisfied. The chain (\ref{chain}) gives the inequality 
$\text{dim }\RZ{(X,x)}\geq d-1$.

 By \cite[Proposition~7.8]{Vaq}, the space $\mathfrak{X}$ has Krull dimension $d$, 
so that we have the inequalities $\text{dim }\RZ{(X,x)}\leq\text{dim }\text{RZ}(X)\leq d$. 
Since $\mathfrak{X}$ is the closure of the trivial valuation, it is an irreducible space 
and as a consequence so is $\text{RZ}(X)$. The fact that $x$ is a closed point implies that 
$\RZ{(X,x)}$ is a proper closed subset of $\text{RZ}(X)$. Therefore the Krull dimension of 
$\RZ{(X,x)}$ must be strictly less than the Krull dimension of $\text{RZ}(X)$. The only 
possibility is then $\text{dim }\RZ{(X,x)}=d-1$ and $\text{dim }\text{RZ}(X)=d$.
\end{proof} 


\subsection{The normalized non-Archimedean link of \texorpdfstring{$\boldsymbol{x}$}{x} in 
\texorpdfstring{$\boldsymbol{X}$}{X}}
\label{subsectionNL}

 For the rest of this section we denote by $|\cdot|_0$ the trivial absolute 
value of $k$. We follow the notations introduced in \cite{FanC}. 

 We associate to $X$ its \emph{analytification} $X^{\text{an}}$ in the sense of 
\cite{Ber}, which is defined as follows. Consider the set of all pairs $(\xi,|\cdot|)$ where 
$\xi$ is a point of $X$ (not necessarily closed) and $|\cdot|$ is an absolute value of the 
residue field $\kappa(\xi)$ of $X$ at $\xi$ extending $|\cdot|_0$. The space $X^{\text{an}}$ 
consists of this set equipped with the weakest topology such that,

\begin{enumerate}[$\bullet$]
\item the natural projection $\imath:X^{\text{an}}\to X$ which maps $(\xi,|\cdot|)$ to $\xi$, 
is continuous, and 
\item for any open subset $U\subseteq X$ and any $f\in\mathcal O_X(U)$, the map 
$\imath^{-1}(U)\to\mathbf R_{\mc}$ which sends $(\xi,|\cdot|)$ to $|f(\xi)|$ 
is continuous, where $f(\xi)$ denotes the residue class of $f$ in $\kappa(\xi)$. 
\end{enumerate}

 The topological space $X^{\text{an}}$ is always Hausdorff, and it is compact 
if and only if $X$ is a complete variety (see \cite[Theorem~3.5.3]{Ber}).

 The \emph{residue field} of $X^{\text{an}}$ at a point $\mathrm x=(\xi,|\cdot|)$ 
is the completion of $\kappa(\xi)$ with respect to the absolute value $|\cdot|$. We denote it 
by $\mathscr H(\mathrm x)$ and its valuation ring by $\mathscr H(\mathrm x)^o$. The inclusion 
$\kappa(\xi)\hookrightarrow\mathscr H(\mathrm x)$ induces a morphism 
$\text{Spec}\:\mathscr H(\mathrm x)\to X$. If it extends to a morphism 
$\text{Spec}\:\mathscr H(\mathrm x)^o\to X$ then we say that $\mathrm x$ has a center in $X$ 
(it follows from the valuative criterion of separateness that this morphism is unique whenever 
it exists). We denote the image of the closed point of $\text{Spec}\:\mathscr H(\mathrm x)^o$ 
under this morphism by $\spe_X(\mathrm x)$ and we call it the \emph{center} of $\mathrm x$ in 
$X$. We denote by $X^\beth$ the set of points in $X^\text{an}$ which have a center in $X$ and 
we endow it with the induced topology. By the valuative criterion of properness, $X^\beth$ and 
$X^\text{an}$ coincide if and only if $X$ is a complete variety (see \cite[Proposition~1.10]{Thu}).

 The \emph{specialization map} $\spe_X:X^\beth\to X$ which sends any point of 
$X^\beth$ to its center in $X$ is an anticontinuous map (i.e.\, the inverse image 
of any open subset of $X$ is closed in $X^\beth$). Recall that the analogous map in the 
Riemann--Zariski setting from $\textnormal{RZ}(X)$ to $X$ is continuous. Furthermore, 
for all $\mathrm x\in X^\beth$ the point $\spe_X(\mathrm x)$ belongs to the closure of 
$\imath(\mathrm x)$ in $X$.

 Let us fix a closed point $x$ of $X$. As we did in the previous section, we 
focus on the fiber of the specialization map above the point $x$. We define 
$\text L(X,x)=\spe_X^{-1}(x)\setminus\set{(x,|\cdot|_0)}$, which is an open subset of 
$X^\beth$, and we endow it with the topology induced by the topology of $X^\beth$. 
The space $\text L(X,x)$ was introduced in \cite{Thu} (in a higher degree of generality) 
where it is called the generic fiber of the formal completion of $X$ along $x$ (in 
\cite{FanC} it is referred to as the \emph{non-Archimedean link of $x$ in $X$}). If $x$ 
is the singular locus of $X$, then \cite[Proposition~4.7]{Thu} shows that $\text L(X,x)$ 
is homotopy equivalent to the dual complex associated to the exceptional divisor of a 
resolution of singularities of $(X,x)$ whose exceptional divisor has simple normal crossings. 
This result holds whenever the base field $k$ is perfect.

 Let $\pi_{X'}:X'\to X$ be a proper birational map. Let us assume that $\pi_{X'}$ induces 
an isomorphism over $X\setminus\set{x}$. Given a point 
$\mathrm x=(\xi,|\cdot|)$ in $\text L(X,x)$, we call $\xi'$ the generic point of the strict 
transform of the subvariety $\overline{\set{\xi}}\subseteq X$ under $\pi_{X'}$. Since 
$\spe_X(\mathrm x)=x$, by the valuative criterion of properness, we have $\mathrm x':=
(\xi',|\cdot|)\in {X'}^\beth$ and $\pi_{X'}(\spe_{X'}(\mathrm x'))=x$. We call 
$\spe_{X'}(\mathrm x')$ the \emph{center} of $\mathrm x$ in $X'$. We obtain in this way an 
anticontinuous map $\spe_{X'}:\text L(X,x)\to X'$, $\mathrm x\mapsto\spe_{X'}(\mathrm x')$.

 We now identify two points of the space $\text L(X,x)$ if they define the same 
valuation. More precisely, we consider the following action of $\mathbf R_{\mec}$ on 
$\text L(X,x)$: given $\lambda\in\mathbf R_{\mec}$ and a point $(\xi,|\cdot|)$ of 
$\text L(X,x)$, we define $\lambda\cdot(\xi,|\cdot|)=(\xi,|\cdot|^\lambda)$. 

\begin{definition}
Given a closed point $x\in X$, the normalized non-Archimedean link of $x$ in $X$, 
denoted by $\NL{(X,x)}$, is the quotient of $\text L(X,x)$ by the group action defined 
above, endowed with the quotient topology. 
\end{definition}

 In \cite{Fan} the space $\NL{(X,x)}$ is introduced for an arbitrary subvariety 
of $X$ and endowed with a richer structure. This space carries a natural analytic structure 
locally modeled on affinoid spaces over $k(\!(t)\!)$. However, these local $k(\!(t)\!)$-analytic 
structures are not canonical and cannot in general be glued to get a global one. We refer 
to \cite{Fan} for details. In this work we only concern ourselves with the topology of 
$\NL{(X,x)}$.

 We start here a series of results that will allow us to describe the space $\NL{(X,x)}$ in 
terms of normalized semivaluations and adopt an additive point of view. With different motivations, 
the space of normalized semivaluations associated to a pair $(X,x)$ was considered before in the 
case of the origin in the complex plane in \cite{VT} and for the germ of a normal surface singularity 
in \cite{CFh}.

 Let $A$ be an integral domain containing the field $k$. A \emph{semivaluation} of $A$ is a 
 map $\nu:A\rightarrow\mathbf{R}\cup\set{+\infty}$ which satisfies $\nu(0)=+\infty$, 
$\nu_{|k^*}=0$ and $\nu(fg)=\nu(f)+\nu(g)$, $\nu(f+g)\geq\min\set{\nu(f),\nu(g)}$ for any 
$f,g\in A$. Note that such a $\nu$ extends to a valuation of the fraction field 
of $A$ if and only if it takes the value $+\infty$ only at zero. In general, the set 
$\sop{\nu}:=\nu^{-1}(+\infty)$ is a prime ideal of $A$ and $\nu$ defines a valuation of the 
fraction field $\kappa(\sop{\nu})$ of the integral domain $A/\sop{\nu}$.

 We say that a semivaluation of A is \emph{centered} if it takes values in $[0,\infty]$. We 
denote by $\Val(A)$ the set of all centered semivaluations of $A$ endowed with the 
\emph{topology of pointwise convergence}. This topology has for a basis of open sets finite 
intersections of subsets of the form $\set{\nu\in\Val(A)\:/\:a<\nu(f)<b}$ where $a$ and $b$ 
are nonnegative real numbers and $f$ belongs to $A$. In other words, it is the topology 
induced by the product topology in $[0,+\infty]^A=\prod_j{Y_j}$, where each $Y_j$ is a copy 
of $[0,+\infty]$ and the product is indexed by $A$. The group $\mathbf R_{\mec}$ now acts 
on $\Val(A)$ by multiplication: $(\lambda\cdot\nu)(f)=\lambda\:\nu(f)$ for all $f\in A$.

\begin{lemma}\label{tychonoff}
The space $\Val(A)$ is compact.
\end{lemma}

\begin{proof}
The space $[0,+\infty]^A$ is Hausdorff and quasi-compact (by Tychonoff's Theo\-rem). Hence 
$\Val(A)$ is Hausdorff, and in order to prove its quasi-compactness it suffices to show 
that it is closed in $[0,+\infty]^A$. In what follows an arrow indicates convergence.

 Let $(\nu_i)_{i\in I}$ be a net in $\Val(A)$ and $\nu\in[0,+\infty]^A$ such that $\nu_i\to\nu$. 
Then we have that $\nu_i(f)\to\nu(f)$ in $[0,+\infty]$ for any $f\in A$. It is straightforward 
to verify that $\nu(0)=+\infty$, $\nu_{|k^*}=0$, and $\nu(fg)=\nu(f)+\nu(g)$ for any $f,g\in A$. 
To conclude that $\Val(A)$ is closed and end the proof, it remains to take $f,g\in A$ and show 
that $\nu(f+g)\geq\min\set{\nu(f),\nu(g)}$. Denote $M_i=\min\set{\nu_i(f),\nu_i(g)}$, $i\in I$. 
The map from $[0,+\infty]\times[0,+\infty]$ to $[0,+\infty]$ defined by 
$(a,b)\mapsto\min\set{a,b}$, is continuous. We have $(\nu_i(f),\nu_i(g))\to(\nu(f),\nu(g))$, 
thus $M_i\to\min\set{\nu(f),\nu(g)}$. Since $\nu_i(f+g)\geq M_i$ for all $i\in I$ 
and $\nu_i(f+g)\to\nu(f+g)$, this yields $\nu(f+g)\geq\min\set{\nu(f),\nu(g)}$.
\end{proof}

 From now on we assume that $A$ is noetherian. Let $\mathfrak m$ be an ideal of $A$. For each 
$\nu\in\Val(A)$ we set $\nu(\mathfrak m)=\min\set{\nu(f)\:/\:f\in\mathfrak m}$. This minimum is 
well--defined. Indeed, if $\set{f_1,\ldots,f_m}$ is a system of generators of $\mathfrak m$, 
then $\nu(\mathfrak m)=\min\set{\nu(f_i)}_{i=1}^m$. Furthermore, the map from $\Val(A)$ to 
$[0,+\infty]$ defined by $\nu\mapsto\nu(f_i)$ is continuous, for $1\leq i\leq m$. Hence we 
also get a continuous map from $\Val(A)$ to $[0,+\infty]$ when sending $\nu$ to $\nu(\mathfrak m)$.

 The \emph{center} of $\nu\in\Val(A)$ is the prime ideal $\set{f\in A\:/\:\nu(f)>0}$ 
of $A$. Note that it contains the ideal $\sop{\nu}$. By the previous paragraph, 
the subset $\Val(A,\mathfrak m)$ of $\Val(A)$ consisting of all semivaluations $\nu$ satisfying 
the condition $\nu(\mathfrak m)=1$, is closed. We endow this subset with the induced topology 
and call it the \emph{space of normalized semivaluations with respect to }$\mathfrak{m}$. 
Since $\Val(A)$ is compact, the space $\Val(A,\frak m)$ is also compact. In addition, 
every semivaluation $\nu$ such that $0<\nu(\frak m)<+\infty$ (i.e.\ such that $\frak m$ is 
contained in the center of $\nu$ but not in $\sop{\nu}$) is proportional to a unique 
normalized semivaluation in $\Val(A,\frak m)$.

\begin{lemma}\label{lam}
Let $\mathcal L(A,\frak m)$ be the set $\set{\nu\in\mathcal V(A)\:/\:0<\nu(\frak m)<+\infty}$ 
equipped with the topology induced by the topology of $\mathcal V(A)$. The map 
\[
\begin{array}{crll}
        \eta: & \mathcal L(A,\frak m) & \rightarrow & \Val(A,\frak m) \\
        & \nu & \mapsto & \frac{\nu}{\nu(\frak m)}
\end{array}
\]
is surjective and continuous. Furthermore, it induces a homeomorphism from the quotient 
of $\mathcal L(A,\frak m)$ by the action of $\mathbf R_{\mec}$ to $\Val(A,\frak m)$.
\end{lemma}

\begin{proof}
The continuity of $\eta$ follows from the continuity of the map from $\Val(A)$ to 
$[0,+\infty]$, $\nu\mapsto\nu(\frak m)$. Observe that $\eta(\nu)=\eta(\nu')$ if and 
only if $\nu$ and $\nu'$ are proportional, so $\eta$ descends to the quotient. We get 
a map $\mathcal L(A,\frak m)/\mathbf R_{\mec}\to\Val(A,\frak m)$ which is continuous and 
bijective. Its inverse map is the composition of the embedding $\Val(A,\frak m)\hookrightarrow
\mathcal L(A,\frak m)$ and the quotient map $\mathcal L(A,\frak m)\to
\mathcal L(A,\frak m)/\mathbf R_{\mec}$.
\end{proof}

 In the case of a maximal ideal, we have the following:

\begin{lemma}\label{semivalOXx}
Suppose that $\frak m$ is a maximal ideal of $A$. Denote by $\jmath:A\to A_{\frak m}$ the 
canonical ring homomorphism. The map 
\[
\begin{array}{crll}
        \tilde\jmath:&\Val(A_{\frak m},\frak m A_{\frak m}) & \rightarrow & \Val(A,\frak m) \\
        &\nu & \mapsto & \nu\circ\jmath
\end{array}
\]
is a homeomorphism.

\end{lemma}

\begin{proof}
Given $\nu\in\mathcal L(A,\frak m)$, the inequalities $0<\nu(\frak m)<+\infty$ and the 
maximality of $\frak m$ imply that $\sop{\nu}\subsetneq\set{f\in A\:/\:\nu(f)>0}=\frak m$. 
In particular, $\nu(g)=0$ for any $g\in A\setminus\frak m$. If $\nu(\frak m)=1$, then 
$\nu':A_{\frak m}\to[0,+\infty]$ defined by $\nu'(f/g)=\nu(f)$ belongs to 
$\Val(A_{\frak m},\frak m A_{\frak m})$ and $\nu'\circ\jmath=\nu$. Hence $\tilde\jmath$ is 
surjective. It is just as straightforward to verify that $\tilde\jmath$ is also injective, 
continuous and with continuous inverse.
\end{proof} 

 The most significant example of $\Val(A,\frak m)$ for us is the \emph{valuative tree} $\Val$ 
of \cite{VT}, which corresponds to $A=\mathbf C[[x,y]]$ and $\frak m=(x,y)$. Let us also 
point out here that we may define a different normali\-zation on 
$\widetilde\Val:=\mathcal L(\mathbf C[[x,y]],(x,y))$. More precisely, we may consider those 
semivaluations of $\widetilde\Val$ taking the value 1 on $z$, where $z\in(x,y)$ and 
$\text{ord}\:z=1$. 
This subset together with the valuation $\text{ord}_z$ and equipped with 
the topology induced by the topology of $\Val(\mathbf C[[x,y]])$ is denoted by $\Val_z$. 
Recall that given a non zero $f\in\mathbf C[[x,y]]$, $\text{ord}_z(f)$ is the largest power 
of $z$ which divides $f$. Observe that $\text{ord}_z\notin\widetilde\Val$ because its center 
is the ideal generated by $z$. The space $\Val_z$ is called \emph{the relative valuative tree 
with respect to }$z$. Such a $z$ gives rise to a semivaluation $\nu_z\in\Val$ by 
considering the intersection multiplicity with $\set{z=0}$ at the origin. The map from $\Val$ 
to $\Val_z$ sending $\nu\neq\nu_z$ to $\nu/\nu(z)$ and $\nu_z$ to $\text{ord}_z$, is a 
homeomorphism (see \cite[Lemma~3.59]{VT}).

 The following is a direct consequence of the previous definitions:

\begin{lemma}\label{terminologycentered}
Suppose that $X$ is an affine variety and denote by $A$ its ring of regular functions. 
The map from $X^\beth$ to $\Val(A)$ sending $(\xi,|\cdot|)\in X^\beth$ to the 
centered semivaluation of $A$ defined by $f\mapsto-\ln|f(\xi)|$ is a homeomorphism.
\end{lemma}

\begin{proof}
A point $\mathrm x=(\xi,|\cdot|)$ of $X^\text{an}$ belongs to $X^\beth$ if and only if 
$|f(\xi)|\leq1$ for all $f\in A$. When this is satisfied, 
$\spe_X(\mathrm x)=\set{f\in A\:/\:|f(\xi)|<1}$. 
The map in the statement is well defined and the center of the image of $\mathrm x$ coincides 
with $\spe_X(\mathrm x)$. The inverse map is defined as follows. Given 
$\nu\in\Val(A)$, abusing notation, we call $\nu$ the valuation induced on $\kappa(\sop{\nu})$. 
It is enough to consider the prime ideal $\sop{\nu}$ of $A$ and the absolute value 
$|\cdot|=e^{-\nu}$ of $\kappa(\sop{\nu})$. Continuity follows directly from the definitions.
\end{proof}

 Let us now come back to our general setting. Let $U=\text{Spec}\:A$ be an open affine 
neighborhood of $x$ in $X$ and let $\frak m$ be the maximal ideal of $A$ corresponding to 
the point $x$. On the one hand, for any $\mathrm x=(\xi,|\cdot|)\in\text L(X,x)$, since 
$\spe_X(\mathrm x)=x$ belongs to the closure of $\xi$ in $X$, the point $\xi$ belongs to 
$U$. Hence the homeomorphism from $U^\beth$ to $\Val(A)$ of Lemma~\ref{terminologycentered} 
restricts to a homeomorphism from $\text L(X,x)$ to $\mathcal L(A,\frak m)$. Then, from the 
definitions of the actions of $\mathbf R_{\mec}$ on these spaces and Lemma~\ref{lam}, we 
get a homeomorphism $\psi$ from $\NL{(X,x)}$ to $\Val(A,\frak m)$. On the other hand, by 
Lemma~\ref{semivalOXx}, we have a homeomorphism $\tilde\jmath:\Val(\mathcal O_{X,x},
\frak m_{X,x})\to\Val(A,\frak m)$. We conclude the following:

\begin{proposition}
The map ${\tilde\jmath}^{-1}\circ\psi:\NL{(X,x)}\to \Val(\mathcal O_{X,x},\frak m_{X,x})$ 
is a homeomorphism.
\end{proposition}

 In what follows we will identify $\NL{(X,x)}$ with the space of normalized semivaluations 
$\Val(\mathcal O_{X,x},\frak m_{X,x})$ and use additive notation. In particular, via this 
identification an element of $\NL{(X,x)}$ is a semivaluation $\nu:\mathcal O_{X,x}\to[0,\infty]$ 
and we are allowed to evaluate it on any $f\in\mathcal O_{X,x}$. We will extensively use this 
fact throughout the rest of this article.

 Finally, we compute the covering dimension of $\NL{(X,x)}$. Recall that the order 
of a family of subsets, not all empty, of a topological space $Z$ is the largest $n$ (if it 
exists) such that the family contains $n+1$ elements with non empty intersection. If such an 
integer $n$ does not exist, then the order is said to be $+\infty$. The \emph{covering dimension} 
of $Z$ is the least $n$ such that any finite open cover of $Z$ has a refinement of order 
not exceeding $n$, or $+\infty$ if there is no such integer.

\begin{remark}
 When $d=1$, any semivaluation of $\NL{(X,x)}$ is in fact a valuation of 
the function field $K$ of $X$ whose valuation ring dominates $\mathcal O_{X,x}$. Thus 
$\NL{(X,x)}$ is a finite space in bijection with $\RZ{(X,x)}$. Since $\NL{(X,x)}$ is 
Hausdorff, it is a discrete topological space and every open cover of the space has a 
refinement consisting of disjoint open sets, so that its covering dimension is zero.
\end{remark}
 
More generally, by using the dimension theory for Berkovich spaces in \cite{Ber}, we 
have:

\begin{proposition}\label{dimNL}
The covering dimension of $\NL{(X,x)}$ is $d-1$.
\end{proposition}

\begin{proof}
Let us denote by $\delta$ the covering dimension of $\NL{(X,x)}$. 

\textit{Step 1: $\delta\leq d-1$.} Every compact space is normal, that is, any two disjoint 
closed subsets of the space have disjoint neighborhoods. Therefore, $\NL{(X,x)}$ is a 
normal space. If $\NL{(X,x)}=\bigcup_{i\geq1}F_i$ where each $F_i$ is a closed subspace of 
covering dimension not exceeding $d-1$, then by \cite[Ch.~3, Theorem~2.5]{Pe} the covering 
dimension of $\NL{(X,x)}$ is not greater than $d-1$. Let us show that such a family 
$\set{F_i}_{i\geq1}$ exists.

 Take an open affine neighborhood $U=\text{Spec}\:A$ of $x$ in $X$ and a system of 
generators $\set{f_1,\dots,f_s}$ of the maximal ideal $\frak m$ corresponding to the point 
$x$. For all $1\leq i\leq s$, we define $F_i=\set{\nu\in\NL{(X,x)}\:/\:\nu(f_i)=1}$. Since 
all the maps $\nu\mapsto\nu(f_i)$ are continuous, $\set{F_i}_{i=1}^s$ is a family of closed 
subset of $\NL{(X,x)}$. We have $\NL{(X,x)}=\bigcup_{i=1}^s F_i$, so in order to conclude 
that $\delta\leq d-1$ it suffices to show that $F_i$ has covering dimension at most $d-1$ 
for all $i\in\set{1,\ldots,s}$.

 Let us take an integer $i$, $1\leq i\leq s$. The regular function $f_i\in A$ induces a 
morphism $U\to\text{Spec}\:k[t]=\mathbf A_k^1$ so that $A$ can be regarded as a $k[t]$-module. 
We set $B=A\otimes_{k[t]}k(\!(t)\!)$. Note that $f_i\otimes1=1\otimes t$ in $B$. Since $A$ 
is a $k$-algebra of finite type, $B$ is a $k(\!(t)\!)$-algebra of finite type. We denote 
$V=\text{Spec}\:B$ and consider its analytification $V^\text{an}$ in the sense of Berkovich. 
By \cite[Theorem~3.4.8(iv)]{Ber}, the covering dimension of $V^\text{an}$ is equal to the 
dimension of $V$, which is $d-1$. Therefore the covering dimension of any closed subspace 
of $V^\text{an}$ is at most $d-1$ (see \cite[Ch.~3, Proposition~1.5]{Pe}). To complete 
the proof we now show that $F_i$ can be identified with a closed subspace of $V^\text{an}$.

 By our choice of notation, $V^\text{an}$ is the set of semivaluations of $B$ which extend 
the valuation $\text{ord}_t$ of $k(\!(t)\!)$, endowed with the topology of pointwise 
convergence. In other words, the underlying set of $V^\text{an}$ consists of all 
semivaluations $\nu$ of $B$ that are trivial on $k^*$ and satisfy $\nu(1\otimes t)=1$. 
We consider the closed subspace 
\[
W=\set{\nu\in V^\beth\:/\:\nu(f_j\otimes1)\geq1\;\forall j\in\set{1,\ldots,s},\:j\neq i}
\]
of $V^\text{an}$. We claim that $W$ is homeomorphic to $F_i$. Indeed, the map $\varphi$ from 
$W$ to $F_i$ which sends $\nu\in W$ to the semivaluation of $A$ defined by $g\mapsto\nu(g\otimes1)$ 
is continuous. Furthermore, $\varphi$ is a bijection. If $\nu$ is a semivaluation of $A$ 
lying in $F_i$ then it extends in a unique way to a semivaluation of $B$. This extension is 
defined by $\widetilde\nu(g\otimes1)=\nu(g)$ and $\widetilde\nu(1\otimes t)=
\widetilde\nu(f_i\otimes1)=1$. Since $W$ is quasi-compact and $F_i$ is Hausdorff, the map 
$\varphi$ is a homeomorphism.

\textit{Step 2: $\delta\geq d-1$, regular case.} 
Assume first that the point $x$ is regular. Since the covering 
dimension is monotone on closed subspaces, it suffices to find a closed subspace of $\NL{(X,x)}$ 
of covering dimension $d-1$. Let $\pi:X'\to X$ be a proper birational map such that $X'$ is 
non-singular, $\pi$ is an isomorphism over $X\setminus\set{x}$ and the exceptional locus 
$E=\text{supp}\:\pi^{-1}(x)$ is a simple normal crossings divisor. 
After a finite sequence of blowing-ups centered at closed points, we may assume 
that there exist $d$ irreducible components $E_1,\ldots,E_d$ of the exceptional locus such 
that $\bigcap_{i=1}^d E_i$ is reduced to a point $p$. We denote by $b_1,\ldots,b_d$ 
their multiplicities. Pick local coordinates $(z_1,\ldots,z_d)$ at the point $p$ such that 
$E_i=\set{z_i=0}$. Since the point $p$ is regular, the completion of the local ring of $X'$ 
at $p$ is isomorphic as $k$-algebra to $k[[z_1,\ldots,z_d]]$. Any 
$\beta=(\beta_1,\dots,\beta_d)\in\mathbf{R}_{\mc}^d$ such that $\sum_{i=1}^d{\beta_i b_i}=1$ 
gives rise to a valuation $\nu_{\mathsmaller\beta}$ of $\NL{(X,x)}$. It suffices to set 
$\nu_{\mathsmaller\beta}(0)=+\infty$ and 
\[
\nu_{\mathsmaller\beta}(f)=\min\set{\sum_{i=1}^d{\beta_i\alpha_i}\:/\:c_\alpha\neq0},
\]
 if $f\in A_\frak m$ is written as $\sum_{\alpha}{c_\alpha z^\alpha}$ 
in $k[[z_1,\ldots,z_d]]$. We call $\Delta$ the simplex 
$$\set{(y_1,\ldots,y_d)\in{\mathbf R}_{\mc}^d\:/\:\sum_{i=1}^d y_i=1}$$ and we define a map 
$\varphi:\Delta\to\NL{(X,x)}$ as follows. Given $y=(y_1,\ldots,y_d)\in\Delta$, we set 
$\varphi(y)=\nu_{\mathsmaller{\beta(y)}}$ where $\beta(y)=(y_1/b_1,\ldots,y_d/b_d)$. The 
map $\varphi$ is injective and continuous (see \cite{MJ} for details). The fact that $\Delta$ 
is quasi-compact implies that $F=\varphi(\Delta)$ is quasi-compact and thus closed in 
$\NL{(X,x)}$. Therefore $\varphi$ yields a homeomorphism between $\Delta$ and a closed 
subspace $F$ of $\NL{(X,x)}$. Since $\Delta$ has covering dimension $d-1$, the subspace 
$F$ has also covering dimension $d-1$. This ends the proof in the regular case.

\textit{Step 3: $\delta\geq d-1$, general case.} We now treat the general case. 
The completion $R$ of the local ring of $X$ at $x$ is an integral 
extension of a formal power series ring $A=k[[z_1,\ldots,z_d]]$. Hence we get a natural 
continuous map $\NL{(X,x)}\to\NL{(\mathbf{A}_k^d,0)}$ by restriction (the passage to formal 
power series is justified in Proposition~\ref{extsemi}, which is independent of this result). 
Let us show that this map is surjective. Observe that this ends the proof since the covering 
dimension of $\NL{(\mathbf{A}_k^d,0)}$ is at least $d-1$ by Step $2$.

Take $\nu\in\NL{(\mathbf{A}_k^d,0)}$. We need to show that $\nu$ extends to a semivaluation 
of $R$. By abuse of notation we also call $\nu$ the rank one valuation it induces on the 
fraction field $H$ of $A/\sop{\nu}$. Let $\mathfrak q$ be a prime ideal of $R$ such that 
$\sop{\nu}=\mathfrak q\cap A$. Note that the extension $A/\sop{\nu}\hookrightarrow R/\mathfrak q$ 
is integral. Let $\hat\nu$ be an extension of $\nu$ to $L$, where $L$ 
denotes the fraction field of $R/\mathfrak q$. On the one hand, the field extension 
$H\hookrightarrow L$ is algebraic so the valuation $\hat\nu$ has rank one. On the other 
hand, we have $R/\mathfrak q\subseteq\overline{A/\sop{\nu}}\subseteq R_{\hat\nu}$, where 
the bar stands for the integral closure in $L$. Therefore $\hat\nu$ induces a centered 
semivaluation of $R$ extending $\nu$. In addition, the fact that the prime ideal 
$\mathfrak m=m_{\hat\nu}\cap R/\mathfrak q$ of $R/\mathfrak q$ intersects $A/\sop{\nu}$ in 
its maximal ideal implies that $\mathfrak m$ is itself maximal. In order to get a 
semivaluation in $\NL{(X,x)}$ one just needs to divide $\hat\nu$ by $\hat\nu(\mathfrak m)$.
\end{proof}

\subsection{The canonical map from RZ\texorpdfstring{$\boldsymbol{(X,x)}$}{(X,x)} to 
NL\texorpdfstring{$\boldsymbol{(X,x)}$}{(X,x)}}
\label{canonicalmap}

 Let $x\in X$ be a closed point. The aim of this subsection is to show that there 
exists a natural surjective continuous map from $\RZ{(X,x)}$ to $\NL{(X,x)}$.

 Given a valuation $\nu\in\RZ{(X,x)}$ of rank $r\geq1$, we consider the 
maximal chain of valua\-tion rings of $K$, 
$R_\nu=R_{\nu_r}\subsetneq R_{\nu_{r-1}}\subsetneq\ldots\subsetneq\ R_{\nu_1}$. 
Let us choose $j$ to be the smallest integer $i\in\set{1,\ldots,r}$ such that 
$\nu_{i}$ is centered in $x$. If $j=1$, then we set 
\[
\pi(\nu)(f)=\frac{\nu_1(f)}{\nu_1(\frak{m}_{X,x})}
\] 
 for any non zero $f\in\mathcal{O}_{X,x}$. Let us now suppose $r>1$ and $j>1$. 
We define then $\pi(\nu)$ as follows.

 The quotient ring $R_{\nu_j}/m_{\nu_{j-1}}$ is a valuation ring of the 
residue field of the valuation $\nu_{j-1}$. It corresponds to the rank one valuation 
$\overline{\nu}_{j-1}$ such that $\nu_j=\nu_{j-1}\circ\overline{\nu}_{j-1}$. By 
definition of the integer $j$, its center in the ring 
$\mathcal{O}_{X,x}/(\mathcal{O}_{X,x}\cap m_{\nu_{j-1}})$ is 
$\frak{m}_{X,x}/(\mathcal{O}_{X,x}\cap m_{\nu_{j-1}})$. Therefore, given a non zero 
$f\in\mathcal{O}_{X,x}$, by setting 
\[
\pi(\nu)(f) =
\begin{cases}
\overline{\nu}_{j-1}(\bar f) & \text{if } f\notin m_{\nu_{j-1}}\\
+\infty & \text{otherwise}
\end{cases}
\]
 where $\bar f$ denotes the residue class of $f$ in 
$\mathcal{O}_{X,x}/(\mathcal{O}_{X,x}\cap m_{\nu_{j-1}})$, we define a semivaluation 
$\pi(\nu)$ on $X$ centered in $x$. After division by a suitable constant we obtain an 
element of $\NL{(X,x)}$ that, by abuse of notation, we will also call $\pi(\nu)$.

 Let us introduce some notations. Given $\nu\in\RZ{(X,x)}$, in the sequel 
we will denote by $\nu_*$ the valuation $\nu_j$ and when $j>1$, by $\nu_*'$ the valuation 
$\nu_{j-1}$ and by $\overline{\nu}_*$ the rank one valuation $\overline{\nu}_{j-1}$.

\begin{proposition}\label{pi}
The map $\pi:\RZ{(X,x)}\rightarrow\NL{(X,x)}$ is surjective and continuous.

\end{proposition}

 We will make use of the following result to prove Proposition~\ref{pi}.

\begin{lemma}\label{septrans}
Let $X$ be an algebraic variety and $Y,Z$ two closed subvarieties of $X$, neither one containing 
the other. Let $X'\to X$ be the blowing-up of $X$ at $Y\cap Z$ (defined by the sum 
$\mathcal I_Y+\mathcal I_Z$ of the ideal sheaves). Then the strict transforms of $Y$ and 
$Z$ have empty intersection. 
\end{lemma}

\begin{proof}
Without loss of generality we may assume that $X$ is affine. Suppose that 
$\mathcal I_Y$ is generated by $\set{f_i}_{i=1}^n$ and $\mathcal I_Z$ by $\set{g_j}_{j=1}^m$. 
The strict transform of $Y$ is empty at a point $x'$ of $X'$ where 
$(\mathcal I_Y+\mathcal I_Z)\cdot\mathcal O_{X',x'}=(f_i)\mathcal O_{X',x'}$ for some $i$. 
Similarly, the strict transform of $Z$ is empty when this ideal is generated by $g_j$ for 
some $j$. Therefore the strict transforms of $Y$ and $Z$ are disjoint. 
\end{proof}

 We are now in position to prove Proposition~\ref{pi}.

\begin{proof}[Proof of Proposition~{\rm\ref{pi}}]
 Let us denote by $R$ the local ring of $X$ at the point $x$ and $\frak m$ 
its maximal ideal. Recall that we call $K$ its fraction field.

\textit{The map $\pi$ is surjective.} Recall that for any $v\in\NL{(X,x)}$, 
$\sop{v}=v^{-1}(+\infty)$ is a prime ideal of $R$. If this ideal is reduced to zero, 
then $v$ extends in a unique way to a rank one valuation of $K$ and the image by $\pi$ 
of this valuation is $v$. Otherwise $v$ induces a valuation $\overline v$ of the fraction 
field of the quotient ring $R/\sop{v}$ whose center in this ring is the maximal ideal 
$\frak m/\sop{v}$. The choice of a valuation $\nu'$ of $K$ such that $R_{\nu'}$ 
dominates the localization $R_\sop{v}$ gives us a composite valuation $\nu'\circ\overline v$ 
of $K$ with center $x$ in $X$ and whose image by $\pi$ is $v$. Therefore $\pi$ is a 
surjective mapping.

\textit{The map $\pi$ is continuous.} To prove that $\pi$ is continuous 
it suffices to show that the inverse image of any open set of the form 
$\set{v\in\NL{(X,x)}\:/\:a<v(f)<b}$, where $a,b\in\mathbf{Q}$, $1\leq a<b$ and 
$f\in\frak m$, is open in $\RZ{(X,x)}$. The openness of such a subset will follow 
once we have proved that the subsets
\begin{align*}
U_{>\alpha}=&\set{\nu\in\RZ{(X,x)}\:/\:\alpha<\pi(\nu)(f)\leq+\infty}\\
\intertext{and}
U_{<\alpha}=&\set{\nu\in\RZ{(X,x)}\:/\:\pi(\nu)(f)<\alpha}
\end{align*}
are open subsets of $\RZ{(X,x)}$ for any rational number $\alpha\geq1$ and any non 
zero $f\in\frak m$. We shall only prove that $U_{>\alpha}$ is open. The same arguments 
show that $U_{<\alpha}$ is also open. In order to prove that $U_{>\alpha}$ is open, 
for any $\nu\in U_{>\alpha}$ we build an open neighborhood $\mathcal U$ of $\nu$ in 
$\RZ{(X,x)}$ contained in $U_{>\alpha}$.

 Pick $\alpha=p/q$ with $p,q$ coprime integers, $p\geq q>0$ and a non zero 
$f\in\frak m$. We denote by $\phi:X'\to\text{Spec}\:R$ the blowing-up of $\text{Spec}\:R$ at 
its closed point. Let $\nu$ be a valuation in $U_{>\alpha}$.

\textit{Case 1: $\pi(\nu)(f)=+\infty$.} We consider the normalized 
blowing-up $\eta:Y\to X'$ of $(\mathfrak{m}^N+(f))\cdot\mathcal{O}_{X'}$, with $N >\alpha$. 
The composed 
morphism $\psi=\phi\circ\eta:Y \to X$ is a proper birational morphism that is an isomorphism 
over $X\setminus\set{x}$ and $\psi^{-1}(x)$ is purely of codimension one. Let 
$\cen:\RZ{(X,x)}\to Y$ be the continuous map which associates to any valuation in $\RZ{(X,x)}$ 
its center in $Y$. 

 Observe that $y=\cen(\nu)$ is contained the strict transform of $\set{f=0}$. 
Indeed, the hypothesis on $\pi(\nu)(f)$ implies that the center $z'$ of $\nu'_*$ in $X'$ 
is contained in $\set{f=0}$. Since $\nu'_*$ is not centered in the point $x\in X$, we 
deduce that $z'$ is not contained in the center of the blowing-up $\eta$. Hence $\cen(\nu'_*)$ 
(which is the strict transform of $z'$) is contained in the strict transform of $\set{f=0}$ 
and therefore $y$ is also contained in that strict transform. 

 Pick any irreducible component $E$ of $\psi^{-1}(x)$ that contains $y$. 
Let $g\in \mathcal{O}_{X,x}$ be such that 
$(g)\cdot\mathcal{O}_{Y,y}=\mathfrak{m}\cdot\mathcal{O}_{Y,y}$. 
Then $(g^N)\cdot\mathcal{O}_{Y,y}= (\mathfrak{m}^N + (f))\cdot \mathcal{O}_{Y,y}$, 
since otherwise $f$ would generate $(\mathfrak{m}^N + (f))\cdot \mathcal{O}_{Y,y}$ 
and this is impossible because the strict transform of $\set{f=0}$ contains $y$. 
We conclude that $g^N$ divides $f$ in $\mathcal{O}_{Y,y}$. Therefore 
$$\nu_E(\mathfrak{m})=\nu_E(g)\leq\frac{\nu_E(f)}{N}<\frac{\nu_E(f)}{\alpha},$$
where $\nu_E$ denotes the divisorial valuation defined by $E$. This means that all 
the irreducible components of $\psi^{-1}(x)$ containing the point $y$ verify 
$\pi(\nu_E)\in U_{>\alpha}$.

 Take an open neighborhood $U\subset Y$ of $y$ which is strictly contained in 
$E_{\mathcal G}\setminus(E_{\mathcal G}\cap D)$, where $E_{\mathcal G}$ is the union 
of all the irreducible components of $\psi^{-1}(x)$ containing $y$ and $D$ is the union 
of the remaining ones. Let us prove that $\mathcal U=\cen^{-1}(U)$ is an open subset 
of $\RZ{(X,x)}$ which satisfies the desired properties. 

 It is an open set since $U$ is open and $\cen$ is continuous, and it contains 
$\nu$ by construction. Finally we show that $\mathcal{U}$ is contained in $U_{>\alpha}$. 
Take $\mu\in\mathcal U$ and set $z=\cen(\mu_*)$. The center of $\mu$ in $Y$ belongs to $U$, 
so that $z$ is also in $U$. Since the center of $\mu_*$ in $X$ is the point $x$, there 
exists $E$ in $E_{\mathcal G}$ such that $z$ belongs to $E$. Pick $g\in\mathcal{O}_{X,x}$ 
such that $(g)\cdot\mathcal{O}_{Y,z}=\mathfrak{m}\cdot\mathcal{O}_{Y,z}$. We have 
$\pi(\nu_E)\in U_{>\alpha}$, so $\nu_E(f^p/g^q)>0$ and $f^p/g^q$ belongs to $\mathfrak{m}_{Y,z}$ 
(note that $f^p/g^q$ is without indeterminacy). In particular, $\mu_*(f^p/g^q)>0$ and we 
deduce that $\pi(\mu)(f)>\alpha$ as required.

\textit{Case 2: $\pi(\nu)(f)<+\infty$.} We replace the birational morphism 
$\eta:Y\to X'$ above by the normalized blowing-up of 
$(\mathfrak{m}^p+(f^q))\cdot\mathcal{O}_{X'}$. Abusing notation, we denote by $\eta$ 
this birational morphism. Write $\psi:Y\to X$ for the composed birational morphism as 
before and $c:\RZ{(X,x)}\to Y$ for the center map. 

 Denote $y_*=\cen(\nu_*)$ and pick $g\in\mathcal{O}_{X,x}$ such that 
$(g)\cdot\mathcal{O}_{Y,y_*}=\mathfrak{m}\cdot\mathcal{O}_{Y,y_*}$. By construction either 
$f^q/g^p$ or $g^p/f^q$ are regular at the point $y_*$. Since $\nu\in U_{>\alpha}$ we deduce 
that $f^q/g^p$ must vanish at $y_*$.

 We say that an irreducible component $E$ of $\psi^{-1}(x)$ is \emph{nice} when the 
image by the map $\pi$ of the divisorial valuation $\nu_E$ defined by $E$ belongs to $U_{>\alpha}$. 
It is \emph{bad} otherwise. We may assume that $y_*$ belongs at least to one nice component. 
Otherwise it suffices to consider the blowing-up $Y'\to Y$ of $Y$ with respect to the sheaf of 
ideals defining $y_*$. The center of $\nu_*$ in $Y'$ must be in the newly created exceptional 
divisor, all of whose components are nice because $f^q/g^p$ has to vanish on them. 

 Let $E_{\mathcal N}$ (resp. $E_{\mathcal B}$) be the union of all nice (resp. bad) 
components containing $y_*$ and $D$ the union of the irreducible components of $\psi^{-1}(x)$ 
which do not contain $y_*$. We denote by $\mathcal{I}_{\mathcal N}$ (resp. 
$\mathcal{I}_{\mathcal B}$) the sheaf of ideals defining $E_{\mathcal N}$ (resp. $E_{\mathcal B}$), 
and for any integer $l\ge1$ we consider the normalized blowing-up $\phi_l:Y_l\to Y$ of the sheaf 
of ideals $\mathcal{I}_{\mathcal N}^l + \mathcal{I}_{\mathcal B}$. 

\textit{Claim:} For $l$ large enough, the center of $\nu_*$ in $Y_l$ does not belong to the 
strict transform of any bad component.

\begin{proof}[Proof of the claim]
 We fix $l\geq 1$, and suppose that the center $z$ of $\nu_*$ in $Y_l$ belongs to 
the strict transform of some bad component. This means that 
$(\mathcal I_{\mathcal N}^l+\mathcal I_{\mathcal B})\cdot\mathcal{O}_{Y_l,z}=
\mathcal I_{\mathcal N}^l\cdot\mathcal{O}_{Y_l,z}$. 
Denote $\mathcal J_{\mathcal B}=\mathcal I_{\mathcal B}\cdot\mathcal{O}_{Y_l,z}$ and 
$\mathcal J_{\mathcal N}=\mathcal I_{\mathcal N}\cdot\mathcal{O}_{Y_l,z}$. This gives 
$\nu_*(\mathcal J_{\mathcal B})\geq l\:\nu_*(\mathcal J_{\mathcal N})>0$. If $\nu_*$ is a 
rank one valuation, then $l$ is bounded and this ends the proof. Suppose that $\nu_*$ has 
rank larger than one. Since $\nu'_*$ is not centered in $x\in X$, we have 
$\nu'_*(\mathcal J_{\mathcal B})=\nu'_*(\mathcal J_{\mathcal N})=0$. This implies 
that $\nu_*(\mathcal J_{\mathcal B})$ and $\nu_*(\mathcal J_{\mathcal N})$ belong to the 
convex subgroup of $\Phi_{\nu_*}$ where $\overline{\nu}_*$ takes its values. This subgroup 
has rank one, so we conclude again that $l$ is bounded.
\end{proof}

 Let us fix an $l$ for which the claim applies. Then we consider the complement 
$U$ in $Y_l$ of the union of the strict transforms of $E_{\mathcal B}$ and $D$. Let 
\[
\mathcal{U}=\set{\mu\in\RZ{(X,x)}\text{ such that the center of }\mu\text{ in }Y_l
\text{ belongs to }U}.
\]
This is an open set since $U$ is open, and it contains $\nu$ since the center of $\nu$ is included 
in the one of $\nu_*$ which belongs to $U$. In order to complete the proof, we now show that 
$\mathcal{U}$ is contained in $U_{>\alpha}$. 

 Take $\mu\in\mathcal U$ and denote by $z$ the center of $\mu_*$ in $Y_l$. The center 
of $\mu$ in $Y_l$ belongs to $U$, so that $z$ is also in $U$. Since the center in $X$ of $\mu_*$ 
is the point $x$, we deduce that either $z$ belongs to the strict transform of $E$ for some 
$E$ in $E_{\mathcal N}$ or $z$ belongs to the exceptional locus of $\phi_l$. Suppose 
that the first happens. Pick $g\in\mathcal{O}_{X,x}$ such that $(g)\cdot\mathcal{O}_{Y_l,z}=
\mathfrak{m}\cdot\mathcal{O}_{Y_l,z}$. Since $E$ is a nice component, $\nu_E(f^p/g^q)>0$ and 
$f^p/g^q$ belongs to $\mathfrak{m}_{Y_l,z}$ (note that $f^p/g^q$ is without indeterminacy). 
Therefore, $\pi(\mu)(f)>\alpha$ as required. If $z$ belongs to the exceptional locus of 
$\phi_l$, then $\cen(\mu_*)$ is $\phi_l(z)\in E_{\mathcal N}$ and we also have 
$\pi(\mu)(f)>\alpha$.
\end{proof}

Let $\pi_{X'}:X'\to X$ be a proper birational map which induces an isomorphism over  
$X\setminus\set{x}$. The map $\spe_{X'}:\text L(X,x)\to X'$ factors through $\NL{(X,x)}$, so 
we obtain an anticontinuous map from $\NL{(X,x)}$ to $X'$. We call again $\spe_{X'}$ this map 
and $\cen_{X'}$ the continuous map which associates to any point of $\RZ{(X,x)}$ its center 
in $X'$. 

\begin{remark}
 Let $\pi_{X'}:X'\to X$ be as in the previous paragraph. Observe that the equality 
$\cen_{X'}=\spe_{X'}\circ\pi$ does not hold in general. Indeed, consider the sequence 
$\nu_1,\ldots,\nu_d$ of valuations in $\RZ{(X,x)}$ built in the proof of Proposition~\ref{dimension}.
We keep the notations introduced there and denote $\nu_d$ by $\nu$. 
The valuation $\nu\in\RZ{(X,x)}$ has rank $d$, hence $\text{tr.deg}_k k_{\nu}=0$ (Abhyankar's 
inequality). As a consequence, the center of $\nu$ in any variety dominating $X$ by a 
proper birational map  is a closed point. However, the center of 
$\pi(\nu)=\frac{\nu_1}{\nu_1(\frak m_{X,x})}$ in $\widetilde X$ is the prime divisor $E_1$.
\end{remark}


\section{Regular case}\label{section2}

This section is devoted to the proof of Theorem~\ref{thmregular} (see Section~\ref{intro}). 
In the first subsection we show the equivalence between assertions $\textup{(2)}$ and $\textup{(3)}$ 
(Theorem~\ref{regNL}) and in the second one, the equivalence between $\textup{(1)}$ and $\textup{(3)}$ 
(Theorem~\ref{firstlemma}).

 For the rest of this section, $x\in X$ is a closed point at which $X$ is analytically 
irreducible. We denote by $R$ the local ring of $X$ at $x$ and by $\frak m$ its maximal ideal. 
By assumption the $\frak m$-adic completion $\widehat{R}$ of $R$ is an integral domain. We call 
$\widehat{\frak m}$ the maximal ideal of $\widehat{R}$ and $\widehat{K}$ its fraction field.

\subsection{Homemorphism type of NL\texorpdfstring{$\boldsymbol{(X,x)}$}{(X,x)}}

The key observation is the following:

\begin{proposition}\label{extsemi}
Let $\imath:R\hookrightarrow\widehat R$ be the natural inclusion. The map 
\[
\begin{array}{crll}
        \hat\imath: & \Val(\widehat R,\widehat{\frak m}) & \rightarrow & \NL{(X,x)} \\
        & \hat\nu & \mapsto & \hat\nu\circ\imath
\end{array}
\]
is a homeomorphism.
\end{proposition}

\begin{proof}
The map $\hat\imath$ is a continuous mapping from a quasi-compact space into a 
Hausdorff space, so it is enough to prove that $\hat\imath$ is a bijection.

 Take $\nu\in\NL{(X,x)}$ and set $\Gamma_\nu=\nu(R)\setminus\set{+\infty}$. Consider a 
nonzero $f\in\widehat{R}$ and a Cauchy sequence $(f_n)_{n=1}^\infty$ in $R$ converging to $f$. 
If the sequence $(\nu(f_n))_{n=1}^\infty$ is not bounded above by an element of $\Gamma_\nu$, then 
we set $\hat\nu(f)=+\infty$. Now suppose that there exists an upper bound in $\Gamma_\nu$ for the 
sequence $(\nu(f_n))_{n=1}^\infty$. Consider the subset  
\[\Lambda=\set{\beta\in\Gamma_\nu\:/\:\forall\:n\:\exists\:n'>n\text{ such that }\nu(f_{n'})\leq\beta}.\] 
It is well known (see \cite[Appendix 3]{ZS}) that since $R$ in noetherian, $\Gamma_\nu$ is well 
ordered. By hypothesis $\Lambda$ is not empty, so we may consider the smallest element $\alpha$ of 
$\Lambda$. If $\alpha=0$ then we deduce that $\nu(f_n)=0$ for all $n$ large enough and we set 
$\hat\nu(f)=0$. Assume that $\alpha>0$. Since $\Gamma_\nu$ does not contain any infinite bounded 
sequence (see \cite[Lemma~3.1]{TeiSemi}), the set $\set{\beta\in\Gamma_\nu\:/\:\beta<\alpha}$ is finite. 
Let $\alpha'\in\Gamma_\nu$ be the immediate predecessor of $\alpha$. By definition of $\alpha$, the 
element $\alpha'$ does not belong to $\Lambda$. Hence $\nu(f_n)>\alpha'$ for all $n$ large enough, 
that is, $\nu(f_n)\geq\alpha$ for $n\gg0$. We deduce that $\nu(f_n)=\alpha$ for all $n$ large enough 
and we set $\hat\nu(f)=\alpha$.

 The definition of $\hat\nu$ does not depend on the choice of the Cauchy sequence. 
Moreover, if $(\nu(f_n))_{n=1}^\infty$ is not bounded then it tends to infinity. It is 
straightforward to verify that $\hat\nu(f)=\lim\limits_{n\to+\infty}\nu(f_n)$ defines a 
semivaluation of $\Val(\widehat R,\widehat{\frak m})$ such that $\hat\imath(\hat\nu)=\nu$. 
In order to end the proof we need to show the uniqueness of $\hat\nu$.

  Take $\hat{\mu}\in\Val(\widehat R,\widehat{\frak m})$ such that $\hat\imath(\hat\mu)=\nu$. 
For any $f\in\widehat R$ we can find a Cauchy sequence $(f_n)_{n=1}^\infty$ in $R$ 
converging to $f$ such that $f-f_n\in\widehat{\frak m}^n$ for any $n\geq1$. Since 
$\hat\mu(\widehat{\frak m})=1$, then for any $n$ we have the inequalities $\hat{\mu}(f-f_n)\geq n$ 
and 
\[\hat{\mu}(f)\geq\min\set{\hat{\mu}(f-f_n),\hat{\mu}(f_n)}\geq\min\set{n,\nu(f_n)}.\]
 If $\hat{\mu}(f)=\alpha\in\mathbf{R}$ then it follows that $\nu(f_n)=\alpha$ 
for every $n>\alpha$. Suppose now that $\hat{\mu}(f)=+\infty$. If the sequence 
$(\nu(f_n))_{n=1}^\infty$ is bounded above then there exists $\beta$ such that 
$\hat{\mu}(f-f_n)=\min\set{\hat{\mu}(f),\nu(f_n)}=\nu(f_n)\leq\beta$ for all $n\geq1$, 
which is a contradiction. Hence $\hat{\mu}(f)=\hat\nu(f)$ for all $f\in\widehat R$.
\end{proof}

We are now in position to prove:

\begin{theorem}\label{regNL}
 Let $X,Y$ be two algebraic varieties defined over the same algebraically closed 
field $k$. For all regular closed points $x\in X$, $y\in Y$, the spaces $\NL{(X,x)}$ and 
$\NL{(Y,y)}$ are homeomorphic if and only if $X$ and $Y$ have the same dimension.
\end{theorem}

\begin{proof}
It follows from Proposition~\ref{dimNL} that $X$ and $Y$ have the same dimension whenever 
$\NL{(X,x)}$ and $\NL{(Y,y)}$ are homeomorphic. Conversely, under our assumptions on 
the points $x$ and $y$, if $X$ and $Y$ have the same dimension then the formal completions 
of the local rings $\mathcal{O}_{X,x}$ and $\mathcal{O}_{Y,y}$ are isomorphic as $k$-algebras. 
Hence $\Val(\widehat{\mathcal{O}}_{X,x},\widehat{\frak m}_{X,x})$ and 
$\Val(\widehat{\mathcal{O}}_{Y,y},\widehat{\frak m}_{Y,y})$ are naturally homeomorphic. 
Proposition~\ref{extsemi} implies that $\NL{(X,x)}$ and $\NL{(Y,y)}$ are homeomorphic.
\end{proof}


\subsection{Homeomorphism type of RZ\texorpdfstring{$\boldsymbol{(X,x)}$}{(X,x)}}

 We now turn to the case of Riemann--Zariski spaces. One should note that, unlike the case 
of $\NL{(X,x)}$, extending a valuation of $\RZ{(X,x)}$ to a valuation of $\widehat K$ whose 
valuation ring dominates $\widehat R$ can not be done in general in an unique way. The approach 
which led to the proof of Theorem~\ref{regNL} seems difficult to carry out successfully in this 
setting. Instead, we focus on the extension of a valuation to the henselization 
$\widetilde R$ of $R$. We are assuming that $\widehat R$ is an integral domain, so $\widetilde R$ 
is also an integral domain. We call $\widetilde{K}$ the fraction field of $\widetilde R$.

 Since $R$ is a local noetherian domain which is excellent we are under the hypothesis 
of the following theorem: 

\begin{theorem}[\cite{HOST}, Theorem~7.1]\label{etaleext}
Let $A$ be a local noetherian excellent domain with fraction field $K$ and $\nu$ a valuation 
of $K$ whose valuation ring dominates $A$. Let $A^e$ be a local \'etale $A$-algebra contained 
in the henselization of $A$. There exists a unique prime ideal $H$ of $A^e$ such that $H\cap A=(0)$ 
and $\nu$ extends to a valuation $\nu^e$ of the fraction field $L$ of $A^e/H$ whose valuation 
ring dominates that ring. Furthermore, $\nu^e$ is the unique extension of $\nu$ to $L$ with this 
property, its group of values is the same as the one of $\nu$, and $H$ is a minimal prime of $A^e$.
\end{theorem}

 Take a valuation $\nu\in\RZ{(X,x)}$. Let us apply the previous result to $\nu$. 
Consider a local \'etale $R$-algebra $R^e$ contained in $\widetilde R$. Since $\widehat R$ is 
a domain, so is $R^e$ and we denote by $K^e$ its fraction field. Since the zero-ideal is 
the unique minimal prime of $R^e$, the last assertion of Theorem~\ref{etaleext} implies 
that $H=(0)$. We conclude that $\nu$ extends in a unique way to a valuation $\nu^e$ of 
$K^e$ whose valuation ring dominates $R^e$. Moreover, $\nu^e$ has the same value group 
as $\nu$.

 We denote by $\RZtilde(X,x)$ the set of all valuations of $\widetilde{K}$ 
whose valuation ring dominates $\widetilde R$, endowed with the topology induced by 
the Zariski topology.

\begin{proposition}\label{passageauhensel}
Let $\rho:K\hookrightarrow\widetilde K$ be the natural inclusion. The map 
\[
\begin{array}{crll}
        \tilde\rho: & \RZtilde(X,x) & \rightarrow & \RZ{(X,x)} \\
        & \widetilde\nu & \mapsto & \widetilde\nu\circ\rho
\end{array}
\]
is a homeomorphism.
\end{proposition}

\begin{proof}
The map $\tilde\rho$ is clearly a continuous map because the valuation ring of 
$\tilde\rho(\widetilde\nu)$ equals $K \cap R_{\widetilde\nu}$ for any 
$\widetilde\nu\in\RZtilde(X,x)$. Let us now show that $\tilde\rho$ is bijective.

 Given $\nu\in\RZ{(X,x)}$, there exists a unique $\widetilde\nu\in\RZtilde(X,x)$ 
such that $\tilde\rho(\widetilde\nu)=\nu$. To see this, take a nonzero $f\in\widetilde R$. 
Since $\widetilde R$ is the inductive limit of the system of equiresidual local \'etale 
$R$-algebras, there exists such a $R$-algebra, say $R^e$, such that $f\in R^e$. We define 
$\widetilde\nu(f)=\nu^e(f)$, where $\nu^e$ is the valuation of the fraction field of $R^e$ whose 
existence Theorem~\ref{etaleext} guarantees. Since $R^e$ is a localization of a finite $R$-algebra, 
$R^e$ is exce\-llent (excellence is preserved by localization and any finitely generated algebra 
over an excellent ring is excellent). If $R^e\hookrightarrow R^{e'}$ then we deduce that 
$\nu^{e'}(g)=\nu^e(g)$ for all $g$ in $R^e$. Therefore $\widetilde\nu$ is well defined and gives 
rise to a valuation of $\RZtilde(X,x)$. The uniqueness of $\widetilde\nu$ follows directly from 
Theorem~\ref{etaleext}.

 In order to complete the proof we need to show that $\tilde\rho$ is an open map. One only 
needs to check that $\tilde\rho(E(f))$ is open in $\RZ{(X,x)}$ for every $f\in\widetilde{K}$. 
Recall that $E(f)$ is the set of all $\mu\in\RZtilde(X,x)$ such that $f\in R_\mu$.

 Pick an element $f\in\widetilde{K}$. Since $K\hookrightarrow \widetilde{K}$ is an 
algebraic field extension, we can consider the minimal polynomial 
$p(t)=t^n+a_{n-1}t^{n-1}+\ldots+a_0\in K[t]$ of $f$. The set $V$ of all valuations 
$\nu\in\RZ{(X,x)}$ such that $a_{i}\in R_\nu$ for all $i\in\set{0,\ldots,n-1}$ is contained 
in $\tilde\rho(E(f))$. Indeed, given $\nu\in V$, if $\widetilde\nu$ is its extension to 
$\RZtilde(X,x)$ then we have the inclusions $R[a_0,\ldots,a_{n-1}]\subseteq R_\nu
\subseteq R_{\widetilde\nu}$. Since $f$ is integral over $R[a_0,\ldots,a_{n-1}]$ this yields 
$f\in R_{\widetilde\nu}$, that is, $\nu\in\tilde\rho(E(f))$. Conversely, 
$\tilde\rho(E(f))\subseteq V$. To see this, let us take $\mu\in\RZtilde(X,x)$ such that 
$f\in R_\mu$ and show that $\tilde\rho(\mu)\in V$. We need to prove that $a_{i}\in R_\mu$ 
for all $i\in\set{0,\ldots,n-1}$.

 Let $L=K(f,\alpha_1,\ldots,\alpha_{n-1})$ be the splitting field of $p(t)$ and 
$\overline{\mu}$ an extension of $\tilde\rho(\mu)$ to $L$. The coefficients of $p(t)$ are 
symmetric polynomials functions of the roots $f,\alpha_1,\ldots,\alpha_{n-1}$, therefore to 
conclude that $a_i\in R_\mu$ for any $i\in\set{0,\ldots,n-1}$ is sufficient to verify 
that $\alpha_j\in R_{\overline{\mu}}$ for all $j\in\set{1,\ldots,n-1}$.

 According to \cite[Ch.~VI \S 7, Corollary~3]{ZS} every extension of $\mu$ to the field $L$ 
can be written as $\overline{\mu}\circ\sigma$ for $\sigma$ in the Galois group $\hbox{Gal}(L|K)$. 
The uniqueness of the extension to $\widetilde{K}$ implies the uniqueness of the extension to the 
subfield $K(f)$, so $\overline{\mu}(\sigma(f))=\mu(f)$ for every $\sigma\in\hbox{Gal}(L|K)$ and 
$\overline{\mu}(\alpha_j)\geq0$ for all $j\in\set{1,\ldots,n}$.
\end{proof}

We now prove the second main result of this section.

\begin{theorem}\label{firstlemma}
Let $X,Y$ be two algebraic varieties defined over the same algebraically closed field $k$. 
For all $x\in X$, $y\in Y$ regular closed points, the spaces $\RZ{(X,x)}$ and $\RZ{(Y,y)}$ 
are homeomorphic if and only if $X$ and $Y$ have the same dimension.
\end{theorem}

\begin{proof}
If $\RZ{(X,x)}$ and $\RZ{(Y,y)}$ are homeomorphic then it follows by Proposition~\ref{dimension} 
that $X$ and $Y$ have the same dimension. Let us now prove the converse.

 Under our assumptions on the points $x$ and $y$, if $X$ and $Y$ have the same dimension 
then the henselizations of $\mathcal O_{X,x}$ and $\mathcal O_{Y,y}$ are isomorphic as 
$k$-algebras. We have then a natural homeomorphism between $\RZtilde(X,x)$ and $\RZtilde(Y,y)$. 
To end the proof it suffices to apply Proposition~\ref{passageauhensel}.
\end{proof}

 When we can freely make use of the existence of resolutions of singularities, 
Theorem~\ref{firstlemma} allows us to show that, in the regular case, the Riemann--Zariski 
space of $X$ at $x$ has a property with the flavour of self-similarity in fractals. More 
precisely, following \cite{selfhomeo} we say that a topological space $Z$ is a 
\emph{self-homeomorphic space}  if for any non empty open subset $U\subseteq Z$ there is 
a subset $V\subseteq U$ such that $V$ is homeomorphic to $Z$.

\begin{corollary}\label{RZisautosim}
Let $X$ be an algebraic variety defined over an algebraically closed field $k$ of characteristic 
zero. If $x\in X$ is a regular closed point, then $\RZ{(X,x)}$ is self-homeomorphic. 
\end{corollary}

\begin{proof}
Suppose that $X$ has dimension $d>1$ (otherwise the result is clear). Theorem~\ref{firstlemma} 
implies that $\RZ{(X,x)}$ is homeomorphic to the Riemann--Zariski space of the $d$-dimensional 
affine space over $k$ at the origin. Therefore it suffices to show that 
$Z=\RZ{(\mathbf A_k^d,0)}$ is self-homeomorphic.

 To see this, take an open subset $U$ of $Z$. Without loss of generality we may assume 
that $U$ is a basic open subset, that is, $U=\set{\nu\in Z\:/\:f_1/g_1,\ldots,f_m/g_m\in R_\nu}$ 
where $f_i$ and $g_i$ are polynomials in $k[x_1,\ldots,x_d]$ and $g_i\neq0$ for all $i=1,\ldots,m$. 
We need to show that there exists $V\subseteq U$ homeomorphic to $Z$. Let $\psi:Y\rightarrow X$ 
be the blowing-up of $\mathbf A_k^d$ with respect to the ideal 
$(x_1,\ldots,x_d)\cdot\prod_{1\leq i\leq m}(f_i,g_i)$ of $k[x_1,\ldots,x_d]$. Pick a resolution 
of singularities $\pi':X'\to Y$ and denote $\pi=\psi\circ\pi'$. We choose a closed point 
$x'\in\pi^{-1}(0)$ in an affine chart $W\subseteq X'$ such that $f_i/g_i\in\mathcal O_{X'}(W)$ 
for all $i=1,\ldots,m$. By construction $\RZ{(X',x')}\subseteq U$. Since $x'$ is regular, using 
again Theorem~\ref{firstlemma} we see that $\RZ{(X',x')}$ is homeomorphic to $Z$. Hence it suffices 
to take $V=\RZ{(X',x')}$ to complete the proof. 
\end{proof}


\section{Graphic tools}\label{section3}

 This section provides a short introduction to trees and graphs. They are both important 
tools in the treatment of the two-dimensional normal case.

\subsection{Trees}\label{subsectrees}

 For us a tree is a rooted non-metric $\mathbf{R}$-tree in the sense of \cite{VT}. We 
refer for details to Sections 3.1 and 7.2 in \textit{loc. cit.} The fact that the last 
condition in the following definition must be explicitly stated was remarked in 
\cite[Definition 3.1]{No}.

 A \emph{tree} is a pair $(\mathcal{T},\leq)$ consisting of a set $\mathcal{T}$ and a partial 
order $\leq$ on $\mathcal{T}$ such that,

\begin{enumerate}[$\bullet$]
\item there exists a unique smallest element $\tau_0$ in $\mathcal{T}$ (called the \emph{root} 
of $\mathcal{T}$);
\item if $\tau\in\mathcal{T}$, then $\set{\sigma\in\mathcal{T}\:/\:\sigma\leq\tau}$ is 
isomorphic (as ordered set) to a real interval;
\item every totally ordered convex subset of $\mathcal{T}$ is isomorphic (as ordered set) 
to a real interval;
\item every non-empty subset of $\mathcal{T}$ admits an infimum in $\mathcal{T}$;
\end{enumerate} 

 In what follows, if $(\mathcal{T},\leq)$ is a tree then $\mathcal T$ is assumed to be equipped 
with the \emph{weak tree topology}, described as follows. Given two elements $\tau,\tau'$ in $\mathcal{T}$, 
we denote by $\tau\wedge\tau'$ the infimum of $\set{\tau,\tau'}$ and we call the subset 
\[[\tau,\tau']=\set{\sigma\in\mathcal{T}\:/\:\tau\wedge\tau'\leq\sigma\leq\tau}
\cup\set{\sigma\in\mathcal{T}\:/\:\tau\wedge\tau'\leq\sigma\leq\tau'}\] 
a \emph{segment}. If $\tau\in\mathcal{T}$, we define an equivalence relation on the set 
$\mathcal{T}\setminus\set{\tau}$ by setting $\tau'\equiv_\tau\tau''$ if and only if 
$[\tau,\tau']\cap[\tau,\tau'']\neq\set{\tau}$. The equivalence classes are called the 
\emph{tangent vectors} at $\tau$ and each of them determines an open subset of $\mathcal{T}$, 
$U_\tau(\tau')=\set{\sigma\in\mathcal{T}\setminus\set{\tau}\::\:\sigma\equiv_\tau\tau'}$. 
The weak tree topology is the topology generated by all these subsets when $\tau$ ranges 
over $\mathcal{T}$. Thus an open subset of $\mathcal{T}$ is a union of finite intersections 
of subsets of the form $U_\tau(\tau')$. 

 We shall use several times the following lemma whose proof is a direct verification 
from the definitions.

\begin{lemma}\label{cambioraiz}
Let $(\mathcal{T},\leq)$ be a tree with root $\tau_0$ and let $\tau_0'\in\mathcal T$. 
Given $\tau,\tau'\in\mathcal T$, define $\tau\leq'\tau'$ if and only if 
$[\tau_0',\tau]\subseteq[\tau_0',\tau']$. Then $(\mathcal{T},\leq')$ is a tree with root 
$\tau_0'$ and segments in $(\mathcal{T},\leq)$ coincide with segments in $(\mathcal{T},\leq')$. 
In particular the weak tree topologies on $\mathcal T$ induced by $\leq$ and $\leq'$ coincide.
\end{lemma}

 Given two different points $\tau,\tau'$ of a tree $(\mathcal{T},\leq)$, for any 
$\sigma\in[\tau,\tau']\setminus\set{\tau,\tau'}$, $U_\sigma(\tau)$ and $U_\sigma(\tau')$ are 
disjoint open neighborhoods of $\tau$ and $\tau'$, thus $\mathcal{T}$ is Hausdorff. Furthermore, 
the segment $[\tau,\tau']$ endowed with the topology induced from that of $\mathcal{T}$ is 
homeomorphic to $[0,1]$ endowed with the induced topology from that of $\mathbf{R}$. Therefore 
any tree is arcwise connected. Moreover it is \emph{uniquely arcwise connected}: 

\begin{lemma}\label{uniqarc}
Let $(\mathcal{T},\leq)$ be a tree and let $\tau$ and $\tau'$ be two different points of 
$\mathcal{T}$. The image of any injective continuous mapping $\gamma:[0,1]\rightarrow\mathcal{T}$ 
with $\gamma(0)=\tau$ and $\gamma(1)=\tau'$ is the segment $[\tau,\tau']$.
\end{lemma}

\begin{proof}
Consider $\gamma([0,1])$ equipped with the induced topology from that of $\mathcal{T}$. Since 
$[0,1]\rightarrow\gamma([0,1])$ is a homeomorphism, if we have $[\tau,\tau']\subseteq\gamma([0,1])$ 
then the inverse image of $[\tau,\tau']$ under $\gamma$ is a connected subset of $[0,1]$ containing $0$ 
and $1$. Hence we have $[\tau,\tau']=\gamma([0,1])$ as desired. Let us therefore show that 
$[\tau,\tau']\subseteq\gamma([0,1])$. 

 By choosing a different partial ordering on $\mathcal T$, we can assume without loss of 
generality that $\tau$ is the root of $\mathcal{T}$ (see Lemma~\ref{cambioraiz}). In view of 
\cite[Corollary~7.9]{VT}, the mapping $f$ which sends $t\in[0,1]$ to 
$\gamma(t)\wedge\tau'\in[\tau,\tau']$ is continuous. Take $\sigma\in[\tau,\tau']$ different 
from $\tau$ and $\tau'$. Define $\Sigma$ to be the set of all $t\in[0,1]$ such that 
$f(t)=\sigma$. On the one hand, $\Sigma$ is nonempty thanks to the Intermediate Value Theorem. 
On the other hand, $\Sigma$ is closed, so $s=\inf\Sigma$ belongs to $\Sigma$. In order to 
complete the proof it suffices to show that $\gamma(s)=\sigma$.

 We proceed now by contradiction. Suppose that 
$\gamma(s)\neq\sigma$. The basic open subset $U_\sigma(\gamma(s))$ of $\mathcal{T}$ is mapped to 
$\sigma$ by $f$, hence $U=U_\sigma(\gamma(s))\cap\gamma([0,1])$ is an open subset of $\gamma([0,1])$ 
containing $\gamma(s)$ such that $f(U)=\set{\sigma}$. It follows that there exists an open neighborhood 
of $s$ in $[0,1]$ whose image by $f$ is reduced to $\sigma$, contradicting the minimality of $s$. 
\end{proof}

More generally, in the sequel we will say that a topological space $Z$ is a \emph{tree} if 
there exists a partial order $\leq$ on $Z$ such that $(Z,\leq)$ is a tree and the weak tree 
topology coincides with the topology of $Z$.

\begin{corollary}\label{subtree}
Any arcwise connected subspace of a tree is a tree.
\end{corollary}

\begin{proof}
Let $(\mathcal{T},\leq)$ be a tree and $\mathcal{S}\subseteq\mathcal{T}$ an arcwise connected subspace 
of $\mathcal{T}$. As previously noted, we can assume that the root $\tau_0$ of $\mathcal{T}$ belongs 
to $\mathcal{S}$. Since $\mathcal{S}$ is arcwise connected, by Lemma~\ref{uniqarc}, the segment 
$[\tau_0,\sigma]$ of $\mathcal{T}$ is contained in $\mathcal{S}$ for any $\sigma\in\mathcal{S}$. It 
is straightforward to deduce from this that $\mathcal{S}$ together with the restriction of the partial 
ordering $\leq$ satisfies the axioms of a tree and that the weak tree topology it carries coincides 
with the topology induced from that of $(\mathcal{T},\leq)$.
\end{proof}

 Let us go back to the valuative tree $\Val$ of Subsection~\ref{subsectionNL}. It is shown 
in \cite[Section~3.2]{VT} (see also \cite{No}) that $\Val$ is a tree with the partial ordering 
$\leq$ defined by $\nu\leq\nu'$ if and only if $\nu(f)\leq\nu'(f)$ for all $f\in\mathbf C[[x,y]]$. 
The proof relies on the encoding of valuations by key polynomials (see \cite[Ch.~2]{VT}). This 
structure was extended in \cite{Gr} to the case of a regular closed point of a surface. In turn, 
the relative valuative tree $\Val_z$ is also a tree. We refer to \cite[Section~3.9]{VT} for details.


\subsection{Graphs}\label{subsecgraphs}

 There are several definitions of a graph. Here we have adopted the view point of \cite{Se}. 
A \emph{graph} $\Gamma$ consists of two sets $V(\Gamma)\neq\emptyset$ and $E(\Gamma)$, whose 
elements are respectively called the \emph{vertices} and the \emph{edges} of $\Gamma$, and 
two maps,

\begin{enumerate}[$\bullet$]
\item $E(\Gamma)\rightarrow E(\Gamma)$, $e\mapsto \bar{e}$, such that $e\neq\bar{e}$ and 
$\bar{\bar{e}}=e$;
\item $E(\Gamma)\rightarrow V(\Gamma)$, $e\mapsto \iota(e)$.
\end{enumerate}

Therefore any edge $e\in E(\Gamma)$ comes with a \emph{reverse edge} $\bar{e}$. For any edge 
$e$ we call $\iota(e)$ and $\iota(\bar{e})$ the \emph{endpoints} of $e$. 
We also say that $e$ is \emph{incident} to $\iota(e)$ and $\iota(\bar{e})$, or $e$ joins $\iota(e)$ 
to $\iota(\bar{e})$. Two vertices are \emph{adjacent} if there exists an edge incident to both 
(a vertex may be adjacent to itself).

 Given $u,v\in V(\Gamma)$, a \emph{path} of length $n\geq1$ joining $u$ to $v$ is a 
sequence of vertices and edges of $\Gamma$ of the form 
$u=v_0,e_1,v_1,e_2,\ldots,e_n,v_n=v$ where $v_{i-1}=\iota(e_i)$ and $v_{i}=\iota(\bar{e}_i)$ for 
$i=1,\ldots,n$. If $e_i\neq\bar{e}_{i+1}$ for $i=1,\ldots,n-1$, then the path is \emph{reduced}. 
By convention we shall call a path of length zero any sequence of the form $u$, where $u$ is a 
vertex of $\Gamma$. A path of length zero is always reduced. A graph is \emph{connected} if any 
two vertices can be joined by a path. Throughout this section by graph we mean a connected graph 
which is in addition \emph{finite}, which means that its sets of vertices and edges are both finite. 

 A graph $\Gamma$ is a combinatorial object, however it can also be considered as a finite 
one-dimensional CW-complex. In order to do this, we endow $V(\Gamma)$ and $E(\Gamma)$ with 
the discrete topology and the unit interval $[0,1]$ with the topology induced from that of 
$\mathbf{R}$. The topological space $|\Gamma|$, which we call the \emph{topological realization} 
of $\Gamma$, is the quotient space of the disjoint union $V(\Gamma)\sqcup(E(\Gamma)\times[0,1])$ 
under the identifications $(e,0)\sim\iota{(e)}$ and $(e,t)\sim(\bar{e},1-t)$ for any $e\in E(\Gamma)$ 
and $t\in[0,1]$. Let us denote by $q$ the quotient map and, for any $e\in E(\Gamma)$, call 
$|e|=q(\set{e}\times[0,1])$ an \emph{edge} of $|\Gamma|$. Then $|e|=|\bar{e}|$ and any edge of 
$|\Gamma|$ is homeomorphic either to $[0,1]$ or the unit circle $\textbf{S}^1$. The \emph{degree} 
of a vertex $v$ of a graph $\Gamma$ is the number of edges $e$ of $\Gamma$ such that $\iota(e)=v$.

 A \emph{morphism} from a graph $\Gamma$ to a graph $\Gamma'$ is a mapping $\gamma$ 
from $V(\Gamma)\cup E(\Gamma)$ to $V(\Gamma')\cup E(\Gamma')$ which sends vertices to vertices 
and edges to edges in such a way that $\gamma(\bar{e})=\overline{\gamma(e)}$ and 
$\gamma(\iota(e))=\iota(\gamma(e))$ for any $e\in E(\Gamma)$. Note that this implies that 
$\gamma(\iota(\bar{e}))=\iota(\overline{\gamma(e)})$ for any $e\in E(\Gamma)$. An \emph{isomorphism} 
of graphs is a bijective morphism of graphs. We say that a graph $\Gamma$ is a \emph{subgraph} 
of $\Gamma'$ if $V(\Gamma)\subseteq V(\Gamma')$, $E(\Gamma)\subseteq E(\Gamma')$ and the inclusion 
$V(\Gamma)\cup E(\Gamma)\hookrightarrow V(\Gamma')\cup E(\Gamma')$ is a morphism. If this is verified, 
then we have a natural closed embedding $|\Gamma|\hookrightarrow|\Gamma'|$.

 A graph $\Gamma$ is a \emph{tree} if given any two vertices $u,v$ of $\Gamma$ there 
exists a unique reduced path joining $u$ to $v$. The topological realization of $\Gamma$ is then 
a tree in the sense introduced before. More precisely, the choice of a vertex of $\Gamma$ determines 
a unique tree structure on $|\Gamma|$:

 Let us choose a vertex of $\Gamma$ and call it $v_0$. Take a point $p=q((e,t))=q((\bar{e},1-t))$ 
of $|\Gamma|$, $p\neq v_0$. There exists a unique edge $b\in\set{e,\bar{e}}$ for which we can find 
a reduced path $v_0,e_1,v_1,\ldots,v_{n-1}=\iota(b),e_n=b,v_n=\iota(\bar{b})$. We may assume that 
$b=e$. Since $\Gamma$ is a tree, we have in addition that any two vertices in the path are different, 
so it induces an injective continuous mapping $\alpha_p:[0,1]\rightarrow|\Gamma|$ such that 
$\alpha_p(0)=v_0$ and $\alpha_p(1)=p$. We set $\alpha_{v_0}(u)=v_0$ for all $u\in[0,1]$. It suffices 
to declare $p\leq p'$ if and only if $\alpha_p([0,1])\subseteq\alpha_{p'}([0,1])$.

 Given a graph $\Gamma$ which is not a tree, following the notations of \cite{Sta} we associate 
to $\Gamma$ its \emph{core} (see also the definition of the skeleton of a quasipolyhedron given 
in \cite[p.~76]{Ber}).

\begin{definition}\label{defcoregraph}
Let $\Gamma$ be a graph. If $\Gamma$ is not a tree, then its core is the subgraph of $\Gamma$ 
obtained by repeatedly deleting a vertex of degree one and the edges incident to it (which are 
exactly two, one being the reverse of the other) until no more vertices of degree one remain. 
We denote the core of $\Gamma$ by $\C{\Gamma}$. By convention we define the core of a tree to 
be the empty set and we set $|\emptyset|:=\emptyset$.
\end{definition}

 We have the following topological characterization:

\begin{lemma}\label{coreversiontop}
Let $\Gamma$ be a graph. The complement of $|\C{\Gamma}|$ in $|\Gamma|$ is the 
set of points $p\in|\Gamma|$ which admit an open neighborhood $U\subsetneq|\Gamma|$ whose closure 
$\overline{U}$ in $|\Gamma|$ is a tree and whose boundary is reduced to a vertex of $\Gamma$.
\end{lemma}

\begin{proof}
 Take $p\in|\Gamma|$. Let us suppose first that $U\subsetneq|\Gamma|$ is an open neighborhood 
of $p$ verifying the hypothesis of the Lemma. Note that $\overline{U}=U\sqcup\set{v}$ for some 
$v\in V(\Gamma)$. Since the boundary of $U$ is reduced to $v$, if $q(\set{e}\times(a,b))$ is 
contained in $U$ for some $0\leq a<b\leq 1$ and $e\in E(\Gamma)$, then $|e|$ is entirely contained 
in $\overline{U}$. From this and the fact that a tree is connected it follows that there exists 
a subgraph $\Gamma'$ of $\Gamma$ such that $\overline{U}=|\Gamma'|$. The graph $\Gamma'$ is clearly 
a tree. Moreover, recall that if $u\in V(\Gamma)$ belongs to $U$ then $|e|\subseteq U$ for any edge 
$e$ of $\Gamma$ such that $\iota(e)=u$ and $\iota(\bar{e})\neq v$. Therefore we have that 
$E(\Gamma')\cap E(\C{\Gamma})=\emptyset$ and $V(\Gamma')\cap V(\C{\Gamma})$ is either empty or equal 
to $\set{v}$, which implies that $p$ does not belong to the topological realization of $\C{\Gamma}$.

 Assume now that $p\notin|\C{\Gamma}|$. Consider 
$e\in E(\Gamma)\setminus E(\C{\Gamma})$ such that $p\in|e|$. There exists a unique subgraph $\Gamma'$ 
of $\Gamma$ such that $e\in E(\Gamma')$, the graph $\Gamma'$ has a unique vertex $v'$ of degree one 
which is in addition an endpoint of $e$, and $\C{\Gamma}=\C{\Gamma'}$. We may suppose that 
$v'=\iota(\bar{e})$. We can find a unique path $v_0,e_1,v_1,\ldots,v_n,e_n=e,v_n=v'$ in $\Gamma'$ of 
length $n\geq1$ where $v_0\in V(\C{\Gamma})$ and $v_i,e_i\notin E(\C{\Gamma})$ for $1\leq i\leq n$. 
The connected component of $|\Gamma|\setminus\set{v_0}$ which contains $|e|$ is an open neighborhood 
of $p$, its closure is a tree and its boundary is the point $v_0$.
\end{proof}

 We may thus think of $\Gamma$ as its core with some disjoint trees attached to it. 
Observe also that $|\Gamma|$ admits a deformation retraction to $|\C{\Gamma}|$. We introduce 
now the equivalence relation in the set of graphs on which the characterization given in 
Theorem~\ref{thmsuperficies} relies.

\begin{definition}\label{defeqgraphs}
Two graphs $\Gamma$ and $\Gamma'$ are equivalent if $|\C{\Gamma}|$ is homeomorphic to 
$|\C{\Gamma'}|$.
\end{definition}

\begin{example}\label{stricterrelation}
 Note that this equivalence relation is stricter than the homotopy equivalence. The 
three graphs consisting of two triangles sharing a vertex, two triangles sharing a side, and 
a line segment with a triangle attached to each endpoint, have all homotopy equivalent topological 
realizations but they are pairwise non-equivalent.
\end{example}

 The equivalence relation of Definition~\ref{defeqgraphs} may also be stated in terms of 
elementary modifications of graphs.

\begin{definition}\label{dualgraph}
We define two kinds of elementary modifications in a fixed graph $\Gamma$:

\begin{enumerate}[$\bullet$]
\item An expansion at a vertex $u\in V(\Gamma)$, that is, adding a new vertex $v$ to 
$V(\Gamma)$ and two edges $e,\bar{e}$ to $E(\Gamma)$ such that $\iota(e)=u$ and $\iota(\bar{e})=v$.
\item A subdivision of an edge $e\in E(\Gamma)$, which consists of the addition of a new 
vertex $v$ to $V(\Gamma)$, the addition of new edges $e',e''$ (and their reverses) to $E(\Gamma)$ 
joining $\iota(e)$ to $v$ and $v$ to $\iota(\bar{e})$ respectively, and the deletion of $e$ and 
$\bar{e}$.  
\end{enumerate}
\end{definition}

 Observe that the subdivision of an edge is an operation on a graph whose result is still a graph 
but which does not induce a morphism. These modifications encode the blowing-up centered at a point 
of a simple normal crossings divisor $E$ on a non singular surface $Y$. 

 More precisely, suppose that the irreducible components of $E$ are smooth and the intersection 
of any two of them is at most a point. To such a divisor $E$ we attach a graph $\Gamma_E$, its 
\emph{dual graph}, whose vertices are in bijection with the irreducible components of $E$ and 
where two vertices are adjacent if and only if the corresponding irreducible components of $E$ 
intersect. Observe that $\Gamma_E$ has no loops ($\iota(e)\neq\iota(\bar{e})$ for any edge $e$) 
and no multiple edges (if two different edges $e,e'$ have the same endpoints then $e'=\bar{e}$).

 Let $\pi:Y'\to Y$ be the blowing-up of $Y$ centered at a point $p$ of $E$. Denote by $E'$ the 
total transform of $E$ under $\pi$. If $p$ is a \emph{free} point, that is, $p$ belongs to a 
single irreducible component $D$ of $E$, then the dual graph $\Gamma_{E'}$ is the result of the 
expansion of $\Gamma_E$ at the vertex associated to $D$. We may define and embedding 
$|\Gamma_E|\hookrightarrow|\Gamma_{E'}|$ and a natural continuous retraction 
$|\Gamma_{E'}|\to|\Gamma_E|$. Otherwise, $p$ is the intersection of two irreducible components 
of $E$ and it is called a \emph{satellite} point. Then $\Gamma_{E'}$ is obtained from $\Gamma_E$ 
by subdividing any of the edges of $\Gamma_E$ associated to $p$. We have that $|\Gamma_E|$ is 
naturally homeomorphic to $|\Gamma_{E'}|$.

 We also consider any isomorphism of graphs as an elementary modification. The following result 
will be useful in Section~\ref{section4}.

\vspace{2cm}

\begin{proposition}\label{eqgraphs}
Let $\Gamma$ and $\Gamma'$ be two graphs. The following are equivalent:
\begin{enumerate}[(1)]
\item The graphs $\Gamma$ and $\Gamma'$ are equivalent.
\item There exist two finite sequences of elementary modifications 
\[\Gamma=\Gamma_0\rightarrow\Gamma_1\rightarrow\ldots\rightarrow\Gamma_n \quad\text{ and }\quad
\Gamma'=\Gamma_0'\rightarrow\Gamma_1'\rightarrow\ldots\rightarrow\Gamma_m'\] such that 
$\Gamma_n$ and $\Gamma_m'$ are isomorphic.
\end{enumerate}
\end{proposition}

\begin{proof}
Let $\Gamma\rightarrow\Delta$ be an elementary modification. Then $\C{\Gamma}=\emptyset$ if 
and only if $\C{\Delta}=\emptyset$. In addition, if $\C{\Gamma}\neq\emptyset$ then the topological 
realizations of $\C{\Gamma}$ and $\C{\Delta}$ are homeomorphic. Hence $\Gamma$ and $\Delta$ are 
equivalent graphs. From this observations we deduce that $\textup{(2)}$ implies $\textup{(1)}$. 
Let us now show the converse. 

Suppose that $\Gamma$ and $\Gamma'$ are equivalent but not isomorphic. If $\C{\Gamma}$ and 
$\C{\Gamma'}$ are both the empty set, then it suffices to consider expansions at vertices to prove the 
result. Otherwise, observe that $\Gamma$ (resp. $\Gamma'$) is obtained from $\C{\Gamma}$ (resp. 
$\C{\Gamma'}$) either after a finite number of expansions at vertices or after an isomorphism 
of graphs. Therefore we may restrict ourselves to the case $\Gamma=\C{\Gamma}$ and $\Gamma'=
\C{\Gamma'}$. We now observe that the topological realizations of any two graphs are homeomorphic 
if and only if they have isomorphic subdivisions. Since $|\Gamma|$ and $|\Gamma'|$ are homeomorphic, 
this ends the proof.
\end{proof}


\section{Normal surface singularity case}\label{section4}

 In this section we study the homeomorphism type of $\RZ{(X,x)}$ and $\NL{(X,x)}$ 
when $(X,x)$ is a normal surface singularity.

\subsection{Largest Hausdorff quotient of RZ\texorpdfstring{$\boldsymbol{(X,x)}$}{(X,x)}}

Let us denote by $\text H$ the largest Hausdorff quotient of $\RZ{(X,x)}$ and by 
$q:\RZ{(X,x)}\to\text H$ the quotient map. If $f$ is a continuous map from $\RZ{(X,x)}$ 
to a Hausdorff topological space $Z$, then by definition of the pair $(\text H,q)$, we have 
a unique continuous map $\tilde f$ from $\text H$ to $Z$ such that $\tilde f\circ q= f$. Observe 
also that, since $q$ is a continuous surjective map and $\RZ{(X,x)}$ is quasi-compact, $\text H$ 
is quasi-compact. Therefore, $\text H$ is compact.

 The purpose of this subsection is to prove the following: 

\begin{proposition}\label{quotientHausdorff}
Let $x$ be a normal point of a surface $X$ and let $\pi$ be the map from $\RZ{(X,x)}$ to 
$\NL{(X,x)}$ defined in Subsection~\ref{canonicalmap}. With the previous notations, 
the map $\tilde\pi:\textnormal H\to\NL{(X,x)}$ is a homeomorphism.  
\end{proposition}

\begin{proof}
The map $\tilde\pi$ is a continuous surjective map from a quasi-compact 
space to a Hausdorff space. Hence to establish the result it is enough to show that 
$\tilde\pi$ is injective. Suppose that $\pi(\nu)=\pi(\nu')=v$ for some $\nu,\nu'\in\RZ{(X,x)}$. 
Let us prove that $q(\nu)=q(\nu')$. 

If $v$ is a valuation of $K$ then by construction $\nu$ and $\nu'$ are both in the closure 
of $v$ in $\RZ{(X,x)}$. Hence $\nu$ and $\nu'$ cannot be separated by disjoint open sets and 
as a consequence they have the same the image under any continuous map to a Hausdorff space. 
We conclude that $q(\nu)=q(\nu')$.

We assume now that $\sop{v}=v^{-1}(+\infty)$ is not reduced to zero. By the hypothesis on the 
dimension, the valuations $\nu$ and $\nu'$ have necessarily rank two. Write $\nu=
\nu_1\circ\overline{\nu}$ and $\nu'=\nu_1'\circ\overline{\nu'}$. Then we have that 
$\sop{v}=\mathcal{O}_{X,x}\cap m_{\nu_1}=\mathcal{O}_{X,x}\cap m_{\nu_1'}
\subsetneq \frak{m}_{X,x}$. The localization 
${\left(\mathcal{O}_{X,x}\right)}_{\sop{v}}$ is a local ring of dimension one with fraction 
field $K$. The fact that $\mathcal{O}_{X,x}$ is integrally closed implies that 
${\left(\mathcal{O}_{X,x}\right)}_{\sop{v}}$ is also integrally closed, and therefore a 
valuation ring of $K$. Since it is dominated by $R_{\nu_1}$ and $R_{\nu_1'}$, we must have 
${\left(\mathcal{O}_{X,x}\right)}_{\sop{v}}=R_{\nu_1}=R_{\nu_1'}$, that is, $\nu_1=\nu_1'$. 
So we conclude that there exists $\mu\in\text{RZ}(X)\setminus\RZ{(X,x)}$ such that $\nu=
\mu\circ\overline{\nu}$ and $\nu'=\mu\circ\overline{\nu'}$. Now it suffices to observe 
that $\pi(\nu)=\pi(\nu')$ means that the valuations $\overline{\nu}$ and $\overline{\nu'}$ 
are the same when restricted to the subring 
$\mathcal{O}_{X,x}/(\mathcal{O}_{X,x}\cap m_\mu)$ of the residue field $k_\mu$. Thus 
$\overline{\nu}=\overline{\nu'}$ and then $\nu=\nu'$.
\end{proof}

 We give next a sketch of a proof that the map $\tilde\pi$ may not be injective in 
higher dimension. The following example was designed in collaboration with Charles Favre.

\begin{example}\label{ejdimtres}
 Consider $R=\mathbf{C}[x_1,x_2,x_3]_{(x_1,x_2,x_3)}$ and denote by $K$ its fraction field. 
Let $\mathbf 0$ be the origin in $\mathbf{A}_\mathbf{C}^3=\text{Spec}~\mathbf{C}[x_1,x_2,x_3]$.
Recall that by the order of a polynomial at $x_i$ we refer to the largest power of $x_i$ 
which divides the polynomial. We define a rank two valuation $\nu_2$ of $K$ by setting 
$$\nu_2(f)=(\nu_1(f),\text{ord}_{x_2}f_1(0,x_2,x_3))\in\mathbf{Z}^2_\text{lex}$$ 
for any nonzero $f\in\mathbf{C}[x_1,x_2,x_3]$, where $\nu_1(f)=\text{ord}_{x_1}f$ and 
$f_1=x_1^{-\nu_1(f)} f$. Similarly, we define $\nu'_1(f)=\text{ord}_{x_2}f$, 
$f_2=x_2^{-\nu'_1(f)} f$ and 
$$\nu'_2(f)=(\nu'_1(f),\text{ord}_{x_1}f_2(x_1,0,x_3))\in\mathbf{Z}^2_\text{lex}.$$ 
Observe that $\nu_2$ and $\nu'_2$ both have center in $R$ the prime ideal $(x_1,x_2)R$, 
and residue fields isomorphic to $\mathbf{C}(x_3)$. Let us denote $\overline{\nu}$ 
the $x_3$-adic valuation of $\mathbf{C}(x_3)$. Then we get two valuations of $K$ 
of rank three, say $\nu=\nu_2\circ\overline{\nu}$ and $\nu'=\nu'_2\circ\overline{\nu}$, 
whose center in $R$ is the maximal ideal of $R$. We have $\nu,\nu'\in
\RZ{(\mathbf{A}_\mathbf{C}^3,\mathbf 0)}$ and by construction $\pi(\nu)=\pi(\nu')$. Indeed, 
their image by $\pi$ is the semivaluation of $R$ which maps $f\in R$ to infinity if 
$f\in (x_1,x_2)R$, and otherwise to $\text{ord}_{x_3}f(0,0,x_3)$.

 Let $\psi:Y\to\mathbf{A}_\mathbf{C}^3$ be the blowing-up of $\mathbf{A}_\mathbf{C}^3$ 
centered at the line $\set{x_1=x_2=0}$ and let $E$ be the exceptional locus of $\psi$ with 
its reduced structure. Both centers $\cen_Y(\nu)$ and $\cen_Y(\nu')$ are closed points of 
$Y$. Indeed $\cen_Y(\nu)=\psi^{-1}(\mathbf 0)\cap\gamma$ and $\cen_Y(\nu')=
\psi^{-1}(\mathbf 0)\cap\gamma'$, where the curves $\gamma$ and $\gamma'$ are the centers 
of $\nu_2$ and $\nu_2'$ in $Y$ respectively. Let us explain briefly how the fact that 
$\gamma\cap\gamma'=\emptyset$ should imply that $\nu$ and $\nu'$ are not identified in the 
largest Hausdorff quotient $\text H$ of  $\RZ{(\mathbf{A}_\mathbf{C}^3,\mathbf 0)}$. Since 
this is not essential to this article, we shall only give a sketch of the proof, leaving 
the details to the interested reader.

Following \cite{Fan} we consider $\NL{(Y,E)}$, which is the following topological space: 
take $\spe_Y^{-1}(E)\setminus(Y^\beth\cap\imath^{-1}(E))$ with the topology induced from the 
topology of $Y^\beth$, and then its quotient space under the standard action of 
$\mathbf R_{\mec}$. The space of germs of formal curves centered at a point in $E$ 
is naturally included in $\NL{(Y,E)}$. 

Generalizing the construction of the map $\pi$ in Subsection~\ref{canonicalmap}, one can 
build a continuous map $f:\RZ{(\mathbf{A}_\mathbf{C}^3,\mathbf 0)}\to\NL{(Y,E)}$. One 
then checks that $\nu$ (resp. $\nu'$) is mapped under $f$ to the formal germ of curve associated 
to $\gamma$ (resp. $\gamma'$) at the point $\cen_Y(\nu)$ (resp. $\cen_Y(\nu')$). This implies 
that $f(\nu)\neq f(\nu')$, and since $\NL{(Y,E)}$ is Hausdorff, we conclude that the valuations 
$\nu$ and $\nu'$ give rise to different points in $\text H$.
\end{example}

The previous example suggests the following question.

\begin{question}
Is the space $\NL(X,x)$ the largest Hausdorff quotient of the projective limit 
$\varprojlim_p p^{-1}(x)$, where $p: Y \to X$ ranges over all proper birational 
morphisms that are \emph{isomorphisms} over $X\setminus\set{x}$?
\end{question}


\subsection{The core of NL\texorpdfstring{$\boldsymbol{(X,x)}$}{(X,x)}}

 Throughout this subsection we suppose that $X$ is a normal algebraic surface and $x$ 
is the only singular point of $X$.

 We say that a proper birational map $\pi_{X'}:X'\rightarrow X$ is a \emph{good resolution} 
if $X'$ is smooth and the exceptional locus $E_{X'}=\pi^{-1}_{X'}(x)_{\text{red}}$ is a divisor 
with normal crossing singularities such that its irreducible components are smooth and the 
intersection of any two of them is at most a point. 

 Since a good resolution $\pi_{X'}:X'\rightarrow X$ is proper and induces an isomorphism from 
the open subset $X'\setminus E_{X'}$ to the open subset $X\setminus\set{x}$, recall that any 
semivaluation of $\NL{(X,x)}$ admits a center in $X'$. Moreover, the map 
$\spe_{X'}:\NL{(X,x)}\to X'$ which sends any semivaluation to its center in $X'$ is anticontinuous.

 To any good resolution we attach a graph, its \emph{dual graph}, which is the dual graph 
$\Gamma_{E_{X'}}$ of $E_{X'}$ (for the definition, see Subsection~\ref{subsecgraphs}). We denote 
$\Gamma_{E_{X'}}$ by $\Gamma_{X'}$.

 The topological realization of any dual graph $\Gamma_{X'}$ can be embedded into 
$\NL{(X,x)}$ as a closed subspace by generalizing the construction given in the Step 2 of 
the proof of Proposition~\ref{dimNL}. We call $i_{X'}:|\Gamma_{X'}|\hookrightarrow\NL{(X,x)}$ 
this embedding and $\Sigma_{X'}$ the image of $|\Gamma_{X'}|$ under $i_{X'}$. Furthermore, 
we may define a natural continuous retraction map $\rec_{X'}:\NL{(X,x)}\rightarrow\esq$. One 
can show that there is actually a strong deformation retraction from $\NL{(X,x)}$ onto $\esq$, 
see \cite{Ber,Fan,Thu}. Keeping the notations introduced in the proof cited above, the retraction 
$\rec_{X'}$ is defined as follows:

 Given $\nu\in\NL{(X,x)}$, if $\nu\in\esq$ then we set $\rec_{X'}(\nu)=\nu$. Suppose that 
$\nu$ does not belong to $\esq$ and denote $p=\spe_{X'}(\nu)$. Observe that 
$p$ is a closed point of $X'$. If $p$ is the intersection of two irreducible components $E_1$ 
and $E_2$ of $E_{X'}$, we take local coordinates $(z_1,z_2)$ at $p$ such that $E_1=\set{z_1=0}$ 
and $E_2=\set{z_2=0}$ and we map $\nu$ to the unique valuation $\nu_t:=\nu_{(t/b_1,(1-t)/b_2)}
\in\esq$ such that $\nu_t(z_1)=\nu(z_1)$ and $\nu_t(z_2)=\nu(z_2)$. Otherwise the point $p$ 
belongs to a single irreducible component $E$ of $E_{X'}$ and we map $\nu$ to the normalized 
divisorial valuation defined by $E$.

 If $\pi_{X''}:X''\rightarrow X$ is a good resolution which dominates 
$\pi_{X'}:X'\rightarrow X$, then $\pi_{X''}$ is obtained by composing $\pi_{X'}$ with a finite 
sequence of point blowing-ups. Then $|\Gamma_{X'}|$ embeds in $|\Gamma_{X''}|$ and we have 
a natural continuous retraction $|\Gamma_{X''}|\to|\Gamma_{X'}|$ (see the discussion before 
Proposition~\ref{eqgraphs}). We have that $\Sigma_{X'}\subseteq\Sigma_{X''}$ and the restriction 
$\rec_{X',X''}$ of $\rec_{X'}$ to $\Sigma_{X''}$ is the natural continuous retraction we may 
consider from $\Sigma_{X''}$ to $\esq$. One can also show that these maps are compatible 
(i.e.\ $\rec_{X'}=\rec_{X',X''}\circ\rec_{X''}$) and the induced continuous mapping 
$\NL{(X,x)}\rightarrow\varprojlim\esq$ is a homeomorphism. For details, see \cite{MJ}.

 The following proposition is a consequence of results of \cite{VT} (see also 
\cite[Proposition~9.5 (i)]{Fan} and the tree structure of Berkovich discs in \cite[Chapter 4]{Ber}).

\begin{proposition}\label{closureUpisatree}
Let $\pi_{X'}:X'\rightarrow X$ be a good resolution and let $E$ be an irreducible component 
of the exceptional locus $E_{X'}$. For any closed point $p\in E$ which is a regular point 
of $E_{X'}$, the closure of $U(p)=\spe_{X'}^{-1}(p)$ in 
$\NL{(X,x)}$ is a tree whose boundary is reduced to the normalized divisorial valuation 
$\nu_E$ associated to $E$.
\end{proposition}

\begin{proof}
The closure $\overline{U}(p)$ of $U(p)$ in $\NL{(X,x)}$ equals $U(p)\bigsqcup\set{\nu_E}$ 
and the boundary of $\overline{U}(p)$ is reduced to the semivaluation $\nu_E$. We now 
prove that $\overline{U}(p)$ is a tree.

 First of all observe that we are not assuming that $\pi_{X'}$ factors through the 
normalized blowing-up of $x\in X$, so we could have some embedded components. Therefore we 
write the pull-back of the coherent sheaf of ideals $\frak m$ of $\mathcal{O}_{X}$ defining 
the point $x$ as 
$\mathcal{O}_{X'}(-C)\otimes_{\mathcal{O}_{X'}}\mathcal{I}$, where $C$ is a divisor on $X'$ 
with $\text{supp }C=E_{X'}$ and $\mathcal{I}$ is a coherent sheaf of ideals in $\mathcal{O}_{X'}$ 
with finite co-support. Choose local coordinates $(z,z')$ at $p$ such that $E=\set{z=0}$. 
The ideal $\mathcal{I}_p$ of $\mathcal{O}_{X',p}$ is either a primary ideal or the ring 
$\mathcal{O}_{X',p}$.

 Suppose that $\mathcal{I}_p$ is a primary ideal of $\mathcal{O}_{X',p}$. Then 
$\frak m_{X,x}\mathcal{O}_{X',p}=(z)^{b_E}\cdot\mathcal{I}_p$. Hence in the open subset 
$U(p)$ the normalization $\nu(\frak m_{X,x})=1$ translates into $\nu(z)b_E+\nu(\mathcal{I}_p)=1$. 
We denote by $\widehat{\mathcal{I}}_p$ the extension of $\mathcal{I}_p$ in $k[[z,z']]$. 
Passing to the completion, we can identify $U(p)$ with the subspace of $\mathcal{V}(k[[z,z']])$ 
consisting of all semivaluations $\nu:k[[z,z']]\rightarrow[0,+\infty]$ whose restriction to $k^*$ 
is trivial, which are centered in the maximal ideal $(z,z')$ and such that 
$\nu(z)b_E+\nu(\widehat{\mathcal{I}}_p)=1$. If $\nu(z)=0$ for some $\nu\in U(p)$, since there 
exists $n\geq1$ such that $z^n\in\mathcal{I}_p$ we would get $0=\nu(z^n)\geq\nu(\mathcal{I}_p)=1$, 
which is a contradiction. Therefore $\nu(z)>0$ for all $\nu\in U(p)$ and we have a well defined 
map $\varphi$ from $\overline{U}(p)$ to the relative valuative tree $\mathcal{V}_z$. It suffices 
to define $\varphi(\nu_E)=\text{ord}_z$ and 
\[\varphi(\nu)=\frac{b_E\:\nu}{1-\nu(\widehat{\mathcal{I}}_p)}=\frac{\nu}{\nu(z)}\] 
for any $\nu\in U(p)$. In the case where $\mathcal{I}_p$ is the ring $\mathcal{O}_{X',p}$, we 
have $\nu(\mathcal{I}_p)=0$ and we may consider $\varphi:\overline{U}(p)\to\mathcal{V}_z$ defined 
exactly as before. We claim that $\varphi$ is a homeomorphism. Indeed, the map from $\mathcal{V}_z$ 
to $\overline{U}(p)$ which sends $\text{ord}_z$ to $\nu_E$ and 
$\nu\in\mathcal{V}_z\setminus\set{\text{ord}_z}$ to $\frac{\nu}{\nu(\frak m_{X,x})}$ is the 
inverse map of $\varphi$. Since $\NL{(X,x)}$ and $\mathcal{V}_z$ are both endowed with the 
topology of pointwise convergence, it is straightforward to verify that they are both continuous 
maps.

 According to \cite[Proposition~3.6.1]{VT}, $\mathcal{V}_z$ is a tree rooted at 
$\text{ord}_z$. From this fact and the existence of $\varphi$ we deduce that $\overline{U}(p)$ 
is a tree (rooted at $\nu_E$) and this finishes the proof.
\end{proof}

The key observation is the following:

\begin{proposition}\label{fibraretrac}
Let $\pi_{X'}:X'\rightarrow X$ be a good resolution. Any fiber $\rec_{X'}^{-1}(\nu)$ 
of the natural retraction 
$\rec_{X'}:\NL{(X,x)}\to\esq$ is a tree whose boundary is reduced to the semivaluation $\nu$.
\end{proposition}

\begin{proof}
Let $\nu$ be a semivaluation in $\esq$. Assume first that $\nu$ is the image under $i_{X'}$ 
of a vertex of $\Gamma_{X'}$ and denote by $E$ the irreducible component of the 
exceptional locus $E_{X'}$ which determines $\nu$. Consider the set $F$ of all closed 
points $p\in E$ which are not singular points of $E_{X'}$. Then 
$\rec_{X'}^{-1}(\nu)=\set{\nu}\bigsqcup_{p\in F} U(p)$, where $U(p)$ is the open subset 
of $\NL{(X,x)}$ of semivaluations whose center in $X'$ is $p$. By 
Proposition~\ref{closureUpisatree}, $\rec_{X'}^{-1}(\nu)$ is the union of the trees 
$\overline{U}(p)=U(p)\sqcup\set{\nu}$, glued along their root $\nu=\nu_E$. In fact, 
the fiber of $\rec_{X'}$ above $\nu$ is itself a 
tree, as we explain next.

 Take $\mu,\mu'\in \rec_{X'}^{-1}(\nu)$. Abusing notation, we declare $\mu\leq\mu'$ 
if there exists $p\in F$ such that $\mu,\mu'\in\overline{U}(p)$ and $\mu\leq\mu'$ in 
$\overline{U}(p)$. This defines a tree structure on $\rec_{X'}^{-1}(\nu)$. Moreover, 
the topology of $\rec_{X'}^{-1}(\nu)$ as subspace of $\NL{(X,x)}$ coincides with the 
weak tree topology induced by $\leq$.

 Suppose now that $\nu=i_{X'}(z)$ where $z$ belongs to the interior of an edge $|e|$ of 
$|\Gamma_{X'}|$. Assume first that $z$ corresponds to an irrational number in the real interval 
$(0,1)$. Then $\nu$ is a quasi-monomial valuation of rational rank two. In particular, 
$\nu$ is not divisorial. As a consequence, 
$\rec_{X',X''}^{-1}(\nu)$ is reduced to $\nu$ and 
$\rec_{X'}^{-1}(\nu)=\rec_{X''}^{-1}(\nu)$ for any  good resolution $\pi_{X''}:X''\to X$ 
dominating $\pi_{X'}$. Since an element of $\NL{(X,x)}$ is determined by its images by 
the retraction maps, we conclude that $\rec_{X'}^{-1}(\nu)=\set{\nu}$. The statement 
is true in this case.

Finally, if $z$ corresponds to a rational number in $(0,1)$ then $\nu$ is a quasi-monomial 
valuation of rational rank one. This means that $\nu$ is a divisorial valuation. Hence there 
exists a finite sequence of blowing-up of points $\pi:X''\rightarrow X'$ such that 
$\pi_{X''}=\pi_{X'}\circ\pi$ is a good resolution and $\spe_{X''}(\nu)$ is a prime 
divisor. Indeed, it is enough to blow-up recursively the center of $\nu$. Observe that all 
the centers of blowing-up are satellite points (see Section~\ref{subsecgraphs}). Otherwise 
$|e|$ viewed in $|\Gamma_{X''}|$ would not be homeomorphic to $(0,1)$. We have that 
$\rec_{X''}^{-1}(\nu)=\rec_{X'}^{-1}(\nu)$, and thus we have reduced the problem to the first 
case we treated above.
\end{proof}

\begin{corollary}\label{caractree}
The normalized non-Archimedean link $\NL{(X,x)}$ is a tree if and only if the dual graph 
associated to any good resolution is a tree.
\end{corollary}

\begin{proof}
Let $\pi_{X'}:X'\rightarrow X$ be a good resolution. Since $\esq\subseteq\NL{(X,x)}$ 
is arcwise connected, if $\NL{(X,x)}$ is a tree then $\esq$ is also a tree by 
Corollary~\ref{subtree}. Let us now show the converse.

 Assume that the dual graph associated to any good resolution is a tree. Take a dual 
graph $\Gamma_{X'}$ and choose a tree structure $(|\Gamma_{X'}|,\leq)$ of $|\Gamma_{X'}|$. 
Proposition~\ref{fibraretrac} allow us to equip $\NL{(X,x)}$ with a tree structure as 
follows. Given two semivaluations $\nu,\nu'$ in $\NL{(X,x)}$ we declare 
$\nu\leq\nu'$ if and only if one of the following conditions is satisfied, 

\begin{enumerate}[$\bullet$]
\item $\nu,\nu'\in\esq$ and $i_{X'}^{-1}(\nu)\leq i_{X'}^{-1}(\nu')$ in $|\Gamma_{X'}|$; 
\item $\nu\in\esq$, $\nu'\notin\esq$ and $i_{X'}^{-1}(\nu)\leq i_{X'}^{-1}(\rec_{X'}(\nu'))$ 
in $|\Gamma_{X'}|$; 
\item $\nu,\nu'\notin\esq$, $\rec_{X'}(\nu)=\rec_{X'}(\nu')$ and 
$\nu\leq\nu'$ in $\rec_{X'}^{-1}(\mu)$ where $\mu=\rec_{X'}(\nu)$ (recall that $\rec_{X'}^{-1}(\mu)$ 
is a tree rooted at $\mu$ by Proposition~\ref{fibraretrac}).
\end{enumerate}

 One can verify that $(\NL{(X,x)},\leq)$ satisfies the four axioms of a tree and 
that the usual topology of $\NL{(X,x)}$ is the weak tree topology induced by $\leq$.
\end{proof}

 We define the \emph{core} of $\NL{(X,x)}$ in a way analogous what to we did for 
graphs (see the topological characterization given in Lemma~\ref{coreversiontop}). In 
\cite[p.~76]{Ber} the core is referred to as the skeleton.

\begin{definition}
The core of the normalized non-Archimedean link $\NL{(X,x)}$ of $x$ in $X$ is the set of all 
semivaluations in $\NL{(X,x)}$ which do not admit a proper open neighborhood whose closure 
is a tree and whose boundary is reduced to a single semivaluation of $\NL{(X,x)}$. We denote 
it $\C{\NL{(X,x)}}$.
\end{definition}

 Observe that by definition $\C{\NL{(X,x)}}$ is empty if and only if $\NL{(X,x)}$ 
is a tree. 

\begin{lemma}\label{coreingraph}
If $\pi_{X'}:X'\rightarrow X$ is a good resolution, then $\C{\NL{(X,x)}}\subseteq\esq$.
\end{lemma}

\begin{proof}
Take $\nu\in\NL{(X,x)}$ and suppose that $\nu\notin\esq$. Then $\rec_{X'}(\nu)$ is 
different from $\nu$. Set $\mu=\rec_{X'}(\nu)$. In view of Proposition~\ref{fibraretrac}, 
$\rec_{X'}^{-1}(\mu)\setminus\set{\mu}$ is an open neighborhood of $\nu$ such that its closure 
$\rec_{X'}^{-1}(\mu)$ is a tree and its boundary is reduced to $\mu$. This means that 
$\nu\notin\C{\NL{(X,x)}}$.
\end{proof}

 However one can be more specific:

\begin{proposition}\label{coresareequal}
Let $\pi_{X'}:X'\rightarrow X$ be a good resolution. The core of $\NL{(X,x)}$ and 
$i_{X'}(|\C{\Gamma_{X'}}|)$ coincide as subspaces of $\NL{(X,x)}$. 
\end{proposition}

\begin{proof}
If $\NL{(X,x)}$ is a tree, then $\C{\NL{(X,x)}}=|\C{\Gamma_{X'}}|=\emptyset$ by 
Corollary~\ref{caractree} and the result follows. Assume now that $\NL{(X,x)}$ is not a tree. 
Observe that under this assumption $|\Gamma_{X'}|$ is not a tree (see Corollary~\ref{caractree}).

 Take $\nu\in\NL{(X,x)}$ and suppose first that 
$\nu\in\C{\NL{(X,x)}}$. We know by Lemma~\ref{coreingraph} that $\nu\in\esq$. 
We proceed by contradiction. Suppose $\nu\notin i_{X'}(|\C{\Gamma_{X'}}|)$. According to 
Lemma~\ref{coreversiontop} there exists an open neighborhood $W\subsetneq|\Gamma_{X'}|$ of 
$i_{X'}^{-1}(\nu)$ such that its closure in $|\Gamma_{X'}|$ is a tree and its boundary is 
reduced to a vertex $v'$ of $\Gamma_{X'}$. Since the retraction $\rec_{X'}$ is 
continuous, $U=\rec_{X'}^{-1}(i_{X'}(W))\subsetneq\NL{(X,x)}$ is an open neighborhood of 
$\nu$. The closure $\overline U$ of $U$ in $\NL{(X,x)}$ equals $\set{\nu'}\bigsqcup U$, 
where $\nu'\in\esq$ corresponds to the vertex $v'$. 
Imitating the proof of Corollary~\ref{caractree}, we see that $\overline U$ in $\NL{(X,x)}$ 
inherits a natural tree structure from that of the closure of $W$ and $\rec_{X'}^{-1}(\mu)$ 
for any $\mu\in W$. Hence $\nu$ does not belong to the core of $\NL{(X,x)}$ and we get a 
contradiction. This proves that $\C{\NL{(X,x)}}\subseteq i_{X'}(|\C{\Gamma_{X'}}|)$.

 In order to finish the proof it suffices to check that $\nu\notin i_{X'}(|\C{\Gamma_{X'}}|)$ 
when $\nu\in\esq$ and $\nu\notin\C{\NL{(X,x)}}$. Suppose that $\nu$ satisfies these two conditions. 
Take a proper open subset $U$ of $\NL{(X,x)}$ which contains $\nu$ and such that its closure 
$\overline{U}$ in $\NL{(X,x)}$ is a tree and its boundary is reduced to a semivaluation $\nu'$. 
Since $\nu\in\esq$, $W=U\cap\esq$ is a non-empty open subset of $\esq$. If $W=\esq$ then 
$\esq\subseteq\overline{U}$ and by Corollary~\ref{subtree}, $\Gamma_{X'}$ would be a tree. 
Thus $W\subsetneq\esq$. Let us denote $Z$ the closure of $W$ in $\esq$. The connectedness of 
$\esq$ implies that $W$ is not closed.  We have $W\subsetneq Z\subseteq\overline{U}\cap\esq=
(U\cap\esq)\bigsqcup(\set{\nu'}\cap\esq)$. We deduce from this that $\nu'$ 
must belong to $\esq$ and $Z=\overline{U}\cap\esq=W\bigsqcup\set{\nu'}$. 

 By enlarging $U$ slightly if necessary, we may choose $\nu'$ such that $v'=i_{X'}^{-1}(\nu')$ 
is a vertex of $\Gamma_{X'}$. If $Z$ is a tree (as subspace of $\esq$) then from 
Lemma~\ref{coreversiontop} it would follow that $\nu\notin i_{X'}(|\C{\Gamma_{X'}}|)$ and this would 
end the proof. Let us prove that $Z$ is a tree.

 The subspace topology that $Z$ inherits from $\esq$ is the same as the one it inherits from 
$\overline{U}$. Since $\overline{U}$ is a tree, if $Z$ is arcwise connected then Corollary~\ref{subtree}
holds and $Z$ is also a tree. Therefore it suffices to take $p\in i_{X'}^{-1}(W)$ 
arbitrary and show that there exists a path $\gamma$ in $i_{X'}^{-1}(Z)$ from $p$ to $v'$. 
Suppose first that $p$ belongs to the interior of an edge $|e|$ of $|\Gamma_{X'}|$. The only 
boundary point of $Z$ is $\nu'$ so that $|e|$ is entirely contained in $i_{X'}^{-1}(Z)$. If the 
edge $e$ is incident to $v'$ then it is easy to define such a path $\gamma$. Otherwise it 
suffices to join either $\iota(e)$ or $\iota(\bar{e})$ to $v'$ by a path in $i_{X'}^{-1}(Z)$. 
Hence we can concentrate on the case when $p$ is a vertex of $\Gamma_{X'}$. Suppose that 
$p\in V(\Gamma_{X'})$. Since the boundary of $Z$ is reduced to $\nu'$, the set 
$\bigcup_{\iota(e)=p}|e|$ must be contained in $i_{X'}^{-1}(Z)$. Note that in particular any 
vertex of $\Gamma_{X'}$ adjacent to $p$ is also in $i_{X'}^{-1}(Z)$. If $p$ is adjacent to 
$v'$ then the edge of $\Gamma_{X'}$ joining $p$ to $v'$ induces the desired path. Otherwise 
the problem is reduced to finding a path in $i_{X'}^{-1}(Z)$ from a vertex adjacent to $p$ 
to $v'$. The graph $\Gamma_{X'}$ is finite and connected, so the existence of such a path 
$\gamma$ is guaranteed. This shows that $Z$ is tree and enables us to complete the proof.
\end{proof}


\subsection{Homeomorphism type of RZ\texorpdfstring{$\boldsymbol{(X,x)}$}{(X,x)} and 
NL\texorpdfstring{$\boldsymbol{(X,x)}$}{(X,x)}}

 The purpose of this subsection is to give the proof of Theorem~\ref{thmsuperficies} (see 
Section~\ref{intro}), which is the main result of this section. We start by presenting some 
lemmas needed for the proof. We end this subsection with an example and a remark.

  In what follows we shall assume that $X$ and $Y$ are algebraic surfaces 
defined over the same algebraically closed field $k$. Recall that, given two regular closed 
points $x\in X$ and $y\in Y$, the choice of an isomorphism between the henselizations of the 
local rings $\mathcal O_{X,x}$ and $\mathcal O_{Y,y}$ gives us a homeomorphism between 
$\RZ{(X,x)}$ and $\RZ{(Y,y)}$ (see Theorem~\ref{firstlemma}).

\begin{lemma}\label{centersarerelated}
Suppose that $X$ and $Y$ are non singular. Let $E,D$ be prime divisors in $X$ and $Y$ 
respectively and let $x\in E$ and $y\in D$ be two regular closed points. Let 
$\sigma:\mathcal{\widetilde{O}}_{Y,y}\rightarrow\mathcal{\widetilde{O}}_{X,x}$ 
be an isomorphism between the henselizations of the local rings which sends an 
equation of $D$ to an equation of $E$. For any valuation $\nu\in\RZ{(X,x)}$, 
$R_{\nu}\subseteq R_{\nu_E}$ if and only if $R_{\varphi(\nu)}\subseteq R_{\nu_D}$, 
where $\varphi:\RZ{(X,x)}\rightarrow\RZ{(Y,y)}$ denotes the homeomorphism induced 
by $\sigma$.
\end{lemma}

\begin{proof}
Let us first consider an arbitrary algebraic variety $X$ defined over $k$ and $x\in X$ a 
regular closed point. Keeping the notations of the proof of Proposition~\ref{passageauhensel}, 
given $\nu$ in $\RZ{(X,x)}$, the valuation $\nu$ and its extension $\widetilde{\nu}\in\RZtilde(X,x)$ 
have the same 
value group $\Phi$. Assume that $\Phi$ has rank greater than one. Then the center $\frak{q}$ 
in $\mathcal{O}_{X,x}$ of the rank one valuation with which $\nu$ is composite coincides with 
$\widetilde{\frak{q}}\cap\mathcal{O}_{X,x}$, where $\widetilde{\frak{q}}$ is the center in 
$\mathcal{\widetilde{O}}_{X,x}$ of the rank one valuation with which $\widetilde{\nu}$ is 
composite. Moreover, $\frak{q}={\frak m}_{X,x}$ if and only if $\widetilde{\frak{q}}$ is the 
maximal ideal of $\mathcal{\widetilde{O}}_{X,x}$. Recall that any prime ideal of height one 
of a UFD is principal. In particular when $\dim\mathcal{O}_{X,x}=2$, we deduce that if $\frak{q}$ 
is generated by an element $f\in\mathcal{O}_{X,x}$ then $\widetilde{\frak{q}}$ is generated by 
an element $\tilde{f}\in\mathcal{\widetilde{O}}_{X,x}$ dividing $f$ in 
$\mathcal{\widetilde{O}}_{X,x}$.

 By hypothesis, we can pick local coordinates $(u,v)$ at $x$ 
and $(u',v')$ at $y$ such that $E=\set{u=0}$, $D=\set{u'=0}$ and $\sigma(u')=u$. 
Let us take a valuation $\nu\in\RZ{(X,x)}$ and suppose that $R_\nu\subseteq R_{\nu_E}$ 
holds. Then $\rank{\nu}=2$ (note that $\nu\neq\nu_E$) and hence $\rank{\varphi(\nu)}=2$. 
We consider $\mu\in\text{RZ}(Y)$ such that $R_{\varphi(\nu)}\subsetneq R_\mu$. Let us 
show that $\mu=\nu_D$. The converse is proved in an analogous way.

 Denote by $\widetilde{\varphi(\nu)}$ the extension of $\varphi(\nu)$ to $\RZtilde(Y,y)$. 
Applying the remarks made at the beginning of the proof to $\nu$, we can write the center 
in $\mathcal{\widetilde{O}}_{Y,y}$ of the rank one valuation with which $\widetilde{\varphi(\nu)}$ 
is composite as $\widetilde{\frak a}=(\sigma^{-1}(\tilde u))\mathcal{\widetilde{O}}_{Y,y}$, 
for some $\tilde u\in\mathcal{\widetilde{O}}_{X,x}$ dividing $u$ in $\mathcal{\widetilde{O}}_{X,x}$. 
In addition, $\widetilde{\varphi(\nu)}$ is not centered in the maximal ideal 
$\widetilde{\frak m}_{Y,y}$. Taking into account again the remarks made at the beginning we see 
that the center of $\mu$ in $\mathcal{O}_{Y,y}$ is $\frak a=\widetilde{\frak a}\cap\mathcal{O}_{Y,y}
\subsetneq\frak m_{Y,y}$. Since $\sigma^{-1}(\tilde u)$ divides $\sigma^{-1}(u)$ in 
$\mathcal{\widetilde{O}}_{Y,y}$ and $u'=\sigma^{-1}(u)$, we deduce that $u'$ belongs 
to $\frak a$. It suffices now to observe that $\frak a$ is a principal ideal and $u'$ is 
irreducible in order to conclude that $\mu=\nu_D$.
\end{proof}

 We might state the following lemma in terms of nets in $\textnormal{RZ}(X)$, but for our 
purposes it suffices to deal with sequences. Note that as a direct consequence of the definition 
of the topology of the Riemann--Zariski space, a sequence of valuations $(\nu_n)_{n=1}^\infty$ 
converges to a valuation $\nu$ (we write $\nu_n\rightarrow\nu$) if and only if for any 
$f\in R_\nu$ there exists $n_0\geq 1$ such that $f\in\bigcap\limits_{n\geq n_0}R_{\nu_n}$.

\begin{lemma}\label{centresdiff}
Suppose that the surface $X$ is non singular. Let $E$ be a prime divisor in $X$ and let 
$(\nu_n)_{n=1}^\infty$ be a sequence of valuations in $\textnormal{RZ}(X)$. If the center 
$x_n$ of $\nu_n$ in $X$ belongs to $E$ for every $n$ and $x_i\neq x_j$ for $i\neq j$, 
then $(\nu_n)_{n=1}^\infty$ is convergent. In addition, the set of all limits of the 
sequence $(\nu_n)_{n=1}^\infty$ is the closure of the divisorial valuation $\nu_E$ 
associated to $E$.
\end{lemma}

\begin{proof}
A sequence $(\nu_n)_{n=1}^\infty$ satisfying the hypothesis of the lemma converges to the divisorial 
valuation $\nu_E$. Indeed, given any $f$ in the function field of $X$ with $\nu_E(f)\geq 0$, for $n$ 
large enough, $x_n$ is not a pole of $f$ and thus $\nu_n(f)\geq0$. This means that $\nu_n\to\nu_E$. 
It follows from simple topological considerations that if $\nu$ is a valuation of $\textnormal{RZ}(X)$ 
in the closure of $\nu_E$, $\nu\neq\nu_E$, then $\nu_n\to\nu$. We will proceed by contradiction to 
prove the converse.

  Take $(\nu_n)_{n=1}^\infty$ satisfying the assumptions about the sequence of centers 
and suppose that $\nu_n\to\nu$ where $\nu\in\textnormal{RZ}(X)$ is not in the closure of $\nu_E$. 
We denote by $x$ the center of $\nu$ in $X$. Note that the continuity of the map which sends a 
valuation of $\text{RZ}(X)$ to its center in $X$ implies that $x_n\rightarrow x$ in $X$. Moreover, 
$x$ must be a closed point of $E$. To see this, observe that either $x$ is a closed point of $X$ 
or the generic point $\xi$ of a prime divisor $D$, $D\neq E$, of $X$. Consider the open subset 
$U=X\setminus E$ of $X$. By hypothesis $x_n\notin U$ for all $n\geq1$. If $x$ is a closed point 
of $X$ and $x\notin E$ then $U$ is an open neighborhood of $x$ in $X$ and this contradicts 
$x_n\to x$. If $x=\xi$ then $x_n\to y$ for all $y\in\overline{\set{\xi}}=D$. Take a closed point 
$y\in D\setminus E$, then $U$ is an open neighborhood of $y$ in $X$ and this contradicts $x_n\to y$.

 Since $\nu$ does not belong to the closure of $\nu_E$ in $\text{RZ}(X)$, it satisfies 
either $\rank{\nu}=1$ or $R_\nu\subsetneq R_{\nu_1}$ for some rank one valuation $\nu_1\in\text{RZ}(X)$, 
$\nu_1\neq\nu_E$. Let us now study both possibilities.

 Suppose that $\nu$ is a rank one valuation. Pick local coordinates $(u,v)$ at $x$ such that 
$E=\set{u=0}$ and a rational function $f$ on $X$ regular at $x$. Since the value group of $\nu$ 
is archimedean, we can find a positive integer $m$ such that $\nu(f^m/u)\geq0$. On the other hand, the 
hypothesis made on the sequence of centers implies that, for $n$ large enough, $f$ is a unit of 
$\mathcal{O}_{X,x_n}$ and therefore $\nu_n(f^m/u)=-\nu_n(u)<0$. We see that $(\nu_n)_{n=1}^\infty$ does 
not converge to $\nu$.

 Now suppose that $\nu$ is a rank two valuation composite with a rank one valuation 
$\nu_1$ different from $\nu_E$. Consider a finite composition $\pi:X'\rightarrow X$ of point blow ups 
above $x$ such that the center $C$ of $\nu_1$ in $X'$ and the strict transform of $E$ are disjoint. 
The sequence $(\pi^{-1}(x_n))_{n=1}^\infty$ of centers in $X'$ does not converge to the center of 
$\nu$ in $X'$, because this center is a closed point of $C$. Hence we conclude that 
$(\nu_n)_{n=1}^\infty$ does not converge to $\nu$ and this ends the proof.
\end{proof}

We are now in position to prove Theorem~\ref{thmsuperficies}.

\begin{proof}[Proof of $\textup{(1)}\Rightarrow\textup{(2)}$]
Assume that $\RZ{(X,x)}$ and $\RZ{(Y,y)}$ are homeomorphic. By Proposition~\ref{quotientHausdorff}, 
$\NL{(X,x)}$ and $\NL{(Y,y)}$ are also homeomorphic.
\end{proof}

\begin{proof}[Proof of $\textup{(2)}\Rightarrow\textup{(3)}$]
Suppose that $\NL{(X,x)}$ and $\NL{(Y,y)}$ are homeomorphic. If $\NL{(X,x)}$ is a tree 
then $\NL{(Y,y)}$ must also be a tree. According to Corollary~\ref{caractree}, $\Gamma_{X'}$ 
and $\Gamma_{Y'}$ are both trees and thus they are equivalent graphs. Suppose that both 
normalized non-Archimedean links are not trees. The definition of the core is purely 
topological, so that we have a natural homeomorphism between the cores of $\NL{(X,x)}$ 
and $\NL{(Y,y)}$ when equipped with their respective induced topologies. Since neither 
$\Gamma_{X'}$ nor $\Gamma_{Y'}$ are trees (again by Corollary~\ref{caractree}) we can 
consider their cores. By Proposition~\ref{coresareequal}, we conclude that 
$|\C{\Gamma_{X'}}|$ and $|\C{\Gamma_{Y'}}|$ are homeomorphic. Therefore $\Gamma_{X'}$ 
and $\Gamma_{Y'}$ are equivalent graphs.
\end{proof}

\begin{proof}[Proof of $\textup{(3)}\Rightarrow\textup{(1)}$]
Suppose that $\Gamma_{X'}$ and $\Gamma_{Y'}$ are equivalent graphs. Our goal is to construct 
a homeomorphism $\varphi$ from $\RZ{(X,x)}$ to $\RZ{(Y,y)}$. We begin by the case where there 
exists an isomorphism of graphs $\tau:\Gamma_{X'}\rightarrow\Gamma_{Y'}$. In Step 1 we assume 
that both exceptional loci, $E:=E_{X'}$ and $D:=E_{Y'}$, are irreducible; while in Step 2 we 
treat the case of any two isomorphic graphs. Next we address the general case.

\textit{Step 1.} 
Let us assume first that $E$ and $D$ are both irreducible. 
Note that the sets underlying $E$ and $D$ have both the same cardinality as the field $k$. Indeed, 
since $E$ is a proper normal curve over $k$, we have a finite flat surjective morphism from $E$ 
to $\mathbf{P}^1_k$ of degree $n=[L:k(t)]$ where $L$ denotes the function field of $E$, and thus 
an injection $E\hookrightarrow\mathbf{P}^1_k\times\set{1,\ldots,n}$. The cardinality of $E$ is 
bounded by the cardinalities of $\mathbf{P}^1_k$ and $\mathbf{P}^1_k\times\set{1,\ldots,n}$, which 
both equal the cardinality of the field $k$.

 We define a bijective map $\varphi:\RZ{(X,x)}\to\RZ{(Y,y)}$ as follows. 
The divisorial valuation associated to $E$ is sent to the divisorial valuation associated 
to $D$, that is, $\varphi(\nu_E)=\nu_D$. We choose a bijection $\sigma$ between the closed 
points of $E$ and those of $D$ and, for every closed point $z\in E$ an isomorphism 
between the henselizations of the local rings 
$\sigma_z:\mathcal{\widetilde{O}}_{Y',\sigma(z)}\rightarrow\mathcal{\widetilde{O}}_{X',z}$ 
which maps the local equation of $D$ to that of $E$. A valuation $\nu\in\RZ{(X',z)}$ is 
sent by $\varphi$ to its image in $\RZ{(Y',\sigma(z))}$ by the homeomorphism induced by 
$\sigma_z$ (see Theorem~\ref{firstlemma}). Let us prove that $\varphi$ is continuous. 
Observe that by construction, $\varphi^{-1}$ will be also continuous.

 According to \cite[Theorem~3.1]{CF}, $\RZ{(X,x)}$ is a Fr\'echet--Urysohn space. 
Thus the continuity of $\varphi$ will follow if, for every sequence of valuations 
$(\nu_n)_{n=1}^\infty$ in $\RZ{(X,x)}$ converging to a valuation $\nu\in\RZ{(X,x)}$, we 
can extract a subsequence such that $(\varphi(\nu_{\gamma(n)}))_{n=1}^\infty$ converges 
to $\varphi(\nu)$. For any positive integer $n$, we denote by $x_n$ the center of $\nu_n$ 
in $X'$. Note that the sequence $(x_n)_{n=1}^\infty$ converges to the center $x'$ of $\nu$ 
in $X'$.

 First suppose that there exist $z\in E$ and $n_0\geq1$ such that $x_n=z$ 
for $n\geq n_0$. If $z$ is a closed point of $E$, then the sequence $(x_n)_{n=1}^\infty$ 
has $z$ as unique limit and therefore $x'=z$. We have that 
$(\nu_n)_{n=n_0}^\infty\subseteq\RZ{(X',z)}$ and $\nu\in\RZ{(X',z)}$. This yields 
$\varphi(\nu_n)\rightarrow\varphi(\nu)$ because $\varphi$ restricted to $\RZ{(X',z)}$ is 
continuous. Suppose now that $z$ is the generic point of $E$, that is, $\nu_n=\nu_E$ for 
all $n\geq n_0$. If moreover $x'$ is the generic point of $E$ then $\nu=\nu_E$ and it is 
clear that $\varphi(\nu_n)\rightarrow\varphi(\nu)$. Otherwise $x'$ is a closed point of $E$ 
and we are then in the situation $R_\nu\subsetneq R_{\nu_E}$. If this is the case, then 
Lemma~\ref{centersarerelated} implies that $R_{\varphi(\nu)}\subsetneq R_{\nu_D}$. Since 
$(\varphi(\nu_n))_{n=1}^\infty$ converges to $\nu_D$, it also converges to any valuation 
in the closure of $\nu_D$, so $\varphi(\nu_n)\rightarrow\varphi(\nu)$. This ends the proof 
in the case where the sequence of centers $(x_n)_{n=1}^\infty$ is stationary.

 Suppose now that sequence of centers does not stabilize and we can extract a 
subsequence $(\nu_{\gamma(n)})_{n=1}^\infty$ of valuations where all the centers are different. 
Then $(\nu_{\gamma(n)})_{n=1}^\infty$ satisfies the assumptions of Lemma~\ref{centresdiff}. Since 
this sequence also converges to $\nu$, the valuation $\nu$ is in the closure of $\nu_E$, and by 
Lemma~\ref{centersarerelated}, $\varphi(\nu)$ is in the closure of $\nu_D$. Observe that by 
construction the centers of $(\varphi(\nu)_{\gamma(n)})_{n=1}^\infty$ are also pairwise distinct. 
Applying again Lemma~\ref{centresdiff} to the sequence $(\varphi(\nu_{\gamma(n)}))_{n=1}^\infty$ 
we conclude that $\varphi(\nu_{\gamma(n)})\rightarrow\varphi(\nu)$. 

 If the sequence of centers does not stabilize but we are not in the previous 
situation, then there exists a finite number of different points $z_1,\ldots,z_l$ of $E$ 
($l\geq 2$) such that $x_n\in\set{z_1,\ldots,z_l}$ for all $n$ large enough and each $z_i$ 
is visited by the sequence infinitely many times. Since $x_n\rightarrow x'$, we deduce that 
$l=2$, one element of $\set{z_1,z_2}$ is the generic point of $E$ and the other one is $x'$ 
(which must be a closed point of $E$). Hence we can extract a subsequence 
$(\nu_{\gamma(n)})_{n=1}^\infty$ of valuations in $\RZ{(X',x')}$ which converges to 
$\nu\in\RZ{(X',x')}$. The continuity of $\varphi$ restricted to $\RZ{(X',x')}$ implies that 
$\varphi(\nu_{\gamma(n)})\rightarrow\varphi(\nu)$. This ends the proof of the continuity of 
$\varphi$ and the proof of Step 1.

\textit{Step 2.} 
Suppose now that $E$ and $D$ both have $m\geq2$ irreducible components and that there exists a 
graph isomorphism $\tau:\Gamma_{X'}\rightarrow\Gamma_{Y'}$. Let $E_1,\ldots,E_m$ be an enumeration 
of the irreducible components of $E$. Then the isomorphism $\tau$ determines an enumeration 
$D_1,\ldots,D_m$ of the irreducible components of $D$. We fix a bijection $\sigma_i$ between the 
closed points of $E_i$ and those of $D_i$, for $1\leq i\leq m$. We do this in such a way that 
$\sigma_i(E_i\cap E_j)=\sigma_j(E_i\cap E_j)=D_i\cap D_j$, for any $i\neq j$ such that 
$E_i\cap E_j\neq\emptyset$. We call $\sigma$ the bijection induced by $\sigma_1,\dots,\sigma_m$ 
between the closed points of $E$ and those of $D$. For any closed point $z$ of $E$, 
we choose an  isomorphism between the henselizations of the local rings 
$\sigma_z:\mathcal{\widetilde{O}}_{Y',\sigma(z)}\rightarrow\mathcal{\widetilde{O}}_{X',z}$ 
that sends the local equation of every $D_i$ passing through $\sigma(z)$ to the local equation 
of the corresponding component $E_i$ in $E$. We define a bijection $\varphi$ from $\RZ{(X,x)}$ 
to $\RZ{(Y,y)}$ exactly as we did before. That is, by means of the homeomorphism at the level 
of valuation spaces determined by each $\sigma_z$ and setting $\varphi(\nu_{E_i})=\nu_{D_i}$ 
for $1\leq i\leq m$.

 In order to check the continuity of $\varphi$ we follow the same idea as in Step 1. 
Let us take a sequence of valuations $(\nu_n)_{n=1}^\infty$ in $\RZ{(X,x)}$ converging to a 
valuation $\nu\in\RZ{(X,x)}$. We denote by $x'$ the center of $\nu$ in $X'$ and by $x_n$ the 
center of $\nu_n$ in $X'$ for any $n\geq 1$. We distinguish again three possibilities for 
the sequence $(x_n)_{n=1}^\infty$ of centers. In fact, we can find $i\in\set{1,\ldots,m}$ such 
that one of the following situations holds,

\begin{enumerate}[$\bullet$]
\item There exists $z\in E_i$ and $n_0\geq1$ such that $x_n=z$ for $n\geq n_0$.
\item We can extract a subsequence of valuations where all the centers in $X'$ are different 
and belong to $E_i$.
\item We can extract a subsequence of valuations where all the centers in $X'$ are equal to 
$x'$ which is in addition a closed point of $E_i$.
\end{enumerate}

 It suffices now to repeat the same arguments used in the proof of the case of one 
prime divisor in each good resolution to show that there exists a subsequence 
$(\nu_{\gamma(n)})_{n=1}^\infty$ such that $(\varphi(\nu_{\gamma(n)}))_{n=1}^\infty$ 
converges to $\varphi(\nu)$.

\textit{Step 3.} 
If $\Gamma_{X'}$ and $\Gamma_{Y'}$ are not isomorphic then by Proposition~\ref{eqgraphs} 
there exist two isomorphic graphs $\Gamma_n$ and $\Gamma'_m$ such that 
$\Gamma_{X'}=\Gamma_0\to\Gamma_1\to\ldots\to\Gamma_n$ and 
$\Gamma_{Y'}=\Gamma'_0\to\Gamma'_1\to\ldots\to\Gamma'_m$, where each arrow denotes an 
elementary modification. 
Let us suppose that one of these sequences is an isomorphism of graphs, for instance the first one. 
Then the second sequence transforms $\Gamma_{Y'}$ into a graph isomorphic to $\Gamma_{X'}$. Recall 
that an elementary modification encodes the blowing-up at a closed point of a simple normal crossings 
divisor on a non singular surface. Then, to $\Gamma'_0\to\Gamma'_1$ we can associate a blowing-up 
$\varphi_1:Y'_1\to Y'$ centered at a closed point of $D$; to $\Gamma'_1\to\Gamma'_2$, a blowing-up 
$\varphi_2:Y'_2\to Y'_1$ centered at a closed point of $\varphi_1^{-1}(D)$; and so on, in such 
a way that $\pi_{Y'}\circ\varphi_1\circ\ldots\circ\varphi_m$ is a good resolution of $(Y,y)$ with 
dual graph isomorphic to $\Gamma_{X'}$. We are now in the case treated above. If neither 
$\Gamma_{X'}\to\Gamma_n$ nor $\Gamma_{Y'}\to\Gamma_m'$ are isomorphisms of graphs, we just need 
to do the previous construction starting from both good resolutions.
\end{proof}

\begin{example}\label{explane} 
Let $(X,x)$ be a rational surface singularity and let $(Y,y)$ be a germ of a cone over an 
elliptic curve, where $X$ and $Y$ are defined over the same algebraically closed field $k$. Then the 
dual graphs associated to the minimal embedded resolutions of $(X,x)$ and $(Y,y)$ are both trees. In 
particular, Theorem~B implies that $\RZ{(X,x)}$, $\RZ{(Y,y)}$, and $\RZ{(\mathbf{A}_k^2,0)}$ are 
homeomorphic. Similarly, $\NL{(X,x)}$, $\NL{(Y,y)}$, and $\NL{(\mathbf{A}_k^2,0)}$ are homeomorphic.
\end{example}

Finally, we observe that the homotopy type of $\NL{(X,x)}$ does not determine its homeomorphism type.

\begin{remark}\label{NLnothomeo}
The homotopy type of $\NL{(X,x)}$ is known to be that of the dual complex associated 
to a log-resolution of the pair $(X,x)$ (see \cite{Fan,Thu}). The equivalence relation that we 
have defined in the set of graphs is stricter than the homotopy equivalence (see 
Example~\ref{stricterrelation}). We may consider two finite connected graphs $\Gamma$ and $\Gamma'$ 
with no vertices of degree one, such that the topological realizations of $\Gamma$ and $\Gamma'$ 
are homotopy equivalent but not homeomorphic. By \cite[Theorem~2]{Kol}, any finite simplicial 
complex of dimension one can be obtained as the dual graph associated to a good resolution of 
an isolated surface singularity. Since $\Gamma$ and $\Gamma'$ are not equivalent, Theorem~B 
also shows that there exist normal surface singularities such that their normalized non-Archimedean 
links are homotopy equivalent but not homeomorphic.
\end{remark}


\end{document}